\documentclass{amsart}
\usepackage[nocompress,noadjust]{cite}
\usepackage{etoolbox}
\usepackage[margin=1in,marginparwidth=0.8in]{geometry}
\usepackage[USenglish]{babel}

\usepackage{tikz}
\usepackage{pgfplots}
\usetikzlibrary{cd}
\usetikzlibrary{babel}
\usepackage{amsmath}
\usepackage{amssymb}
\usepackage{amsthm}
\usepackage{amscd}
\usepackage{enumitem}
\usepackage[pdfusetitle,unicode,hidelinks]{hyperref}
\usepackage{bbm}
\usepackage{etoolbox}

\usepackage[utf8]{inputenc}
\usepackage[T1]{fontenc}

\usepackage[depth=2]{bookmark}
\usepackage[nameinlink]{cleveref}
\crefformat{section}{\S#2#1#3}

\newtheorem{proposition}{Proposition}
\newtheorem{corollary}[proposition]{Corollary}
\newtheorem{lemma}[proposition]{Lemma}
\newtheorem{theorem}[proposition]{Theorem}
\newtheorem{conjecture}[proposition]{Conjecture}
\newtheorem*{conjecture*}{Conjecture}
\newtheorem*{theorem*}{Theorem}
\newtheorem*{corollary*}{Corollary}
\newtheorem*{proposition*}{Proposition}
\newtheorem*{lemma*}{Lemma}
\theoremstyle{definition}
\newtheorem{definition}[proposition]{Definition}
\newtheorem{construction}[proposition]{Construction}

\newtheorem*{definition*}{Definition}
\newtheorem*{construction*}{Construction}
\theoremstyle{remark}
\newtheorem{remark}[proposition]{Remark}

\newtheorem*{remark*}{Remark}
\newtheorem*{variant*}{Variant}

\newtheorem{example}[proposition]{Example}
\newtheorem*{example*}{Example}

\newcommand{\id}{\operatorname{id}}
\newcommand{\Z}{\mathbb{Z}}
\def\C{\mathbb C}

\newcommand{\Q}{\mathbb{Q}}
\newcommand{\F}{\mathbb{F}}

\newcommand{\1}{\mathbbm{1}}

\newcommand{\Spt}{\mathcal{S}\mathrm{pt}{}}

\newcommand{\Perf}{\mathrm{Perf}}

\DeclareMathOperator*{\colim}{colim}

\let\lim=\relax
\DeclareMathOperator*{\lim}{lim}
\def\Map{\mathrm{Map}}

\def\map{\mathrm{Map}}
\def\CAlg{\mathrm{CAlg}}

\def\Ind{\mathrm{Ind}}
\def\Span{\mathrm{Span}}
\def\Cat{\mathcal{C}\mathrm{at}{}}
\def\Spc{\mathcal{S}\mathrm{pc}{}}
\def\Fin{\cat F\mathrm{in}}
\def\Fun{\mathrm{Fun}}

\newcommand{\tr}{\mathrm{tr}}

\def\op{\mathrm{op}}

\let\cat=\mathrm

\def\FEt{\mathrm{FEt}{}}

\def\mot{\mathrm{mot}}
\newcommand{\et}{{\acute{e}t}}

\tikzset{
	symbol/.style={
		draw=none,
		every to/.append style={
			edge node={node [sloped, allow upside down, auto=false]{$#1$}}}
	}
}

\newcommand{\KO}{\mathrm{KO}}

\newcommand{\K}{\mathrm{K}}
\renewcommand{\H}{\mathrm{H}}

\newcommand{\comp}{{{\kern -.5pt}\wedge}}

\newcommand{\Dbb}{\mathbb{D}}

\usepackage{relsize}
\usepackage[bbgreekl]{mathbbol}
\usepackage{amsfonts}

\DeclareSymbolFontAlphabet{\mathbb}{AMSb} 
\DeclareSymbolFontAlphabet{\mathbbl}{bbold}

\numberwithin{proposition}{section}
\numberwithin{equation}{section}

\theoremstyle{proposition}

\usepackage{tabu, adjustbox}
\let\amp=&
\usetikzlibrary{patterns}
\usetikzlibrary{calc}
\newcommand{\x}{\times}
\newcommand{\GL}{\mathrm{GL}}

\newcommand{\image}{\mathrm{Im}}
\newcommand{\wt}[1]{\widetilde{#1}}

\newcommand{\Cbb}{\mathbb{C}}
\newcommand{\Rbb}{\mathbb{R}}
\newcommand{\Cx}{\Cbb^\x}
\newcommand{\Zp}{\Z_p}
\newcommand{\Zell}{\Z_\ell}

\newcommand{\Zpx}{\Zp^\x}

\newcommand{\znx}{\left(\Z/N\right)^\times}

\newcommand{\Fbb}{\mathbb{F}}

\newcommand{\Fq}{\Fbb_q}

\newcommand{\Qell}{\Q_\ell}
\newcommand{\Fr}{\mathrm{Fr}}
\newcommand{\sX}{\mathcal{X}}
\newcommand{\piet}{\pi_1^\et}

\newcommand{\Ocal}{\mathcal{O}}

\newcommand{\gal}{\mathrm{Gal}}
\newcommand{\aut}{\mathrm{Aut}}
\newcommand{\norm}{\mathrm{Norm}}

\newcommand{\ord}{\mathrm{ord}}
\newcommand{\Frac}{\mathrm{Frac}}
\newcommand{\Dcal}{\mathcal{D}}
\newcommand{\KU}{\mathrm{KU}}
\newcommand{\BU}{\mathrm{BU}}
\newcommand{\E}{\mathrm{E}}

\newcommand\rcolon{
	\nobreak
	\mspace{6mu plus 1mu}
	{:}
	\nonscript\mkern-\thinmuskip
	\mathpunct{}
	\mspace{2mu}
}

\DeclareMathOperator{\spec}{Spec}

\newcommand{\M}{\mathrm{M}}

\newcommand{\Ical}{\mathcal{I}}

\newcommand{\Fcal}{\mathcal{F}}
\usepackage{stmaryrd}

\newcommand{\llb}{\llbracket}
\newcommand{\rrb}{\rrbracket}

\newcommand{\simto}{\overset{\sim}{\longrightarrow}}
\newcommand{\Orb}{\Ocal\mathrm{rb}}

\setlist{leftmargin=*}
\allowdisplaybreaks
\newcommand{\hemail}[1]{\email{\href{mailto:#1}{#1}}}
\title{Equivariant algebraic $\K$-theory and Artin $L$-functions}

\author{Elden Elmanto}
\address{Department of Mathematics, University of Toronto, Toronto, ON, M5S 2E4, Canada}
\hemail{elden.elmanto@utoronto.ca}

\author{Ningchuan Zhang}
\address{Department of Mathematics, Indiana University Bloomington, Bloomington, IN 47405, USA}
\address{Max Planck Institute for Mathematics, Bonn 53111, Germany}
\hemail{nz7@iu.edu}

\begin{document}

	\begin{abstract}
		In this paper, we generalize the Quillen--Lichtenbaum Conjecture relating  special values of Dedekind zeta functions to algebraic $\K$-groups. The former has been settled by Rost--Voevodsky up to the Iwasawa Main Conjecture. Our generalization extends the scope of this conjecture to Artin $L$-functions of Galois representations of finite, function, and totally real number fields. The statement of this conjecture relates  norms of the special values of these $L$-functions to sizes of equivariant algebraic $\K$-groups with coefficients in an equivariant Moore spectrum attached to a Galois representation.
		
		We prove this conjecture in many cases, integrally, except up to a possible factor of powers of $2$ in the non-abelian and  totally real number field case. In the finite field case, we further determine the group structures of their equivariant algebraic $\K$-groups with coefficients in Galois representations.   At heart, our method lifts the M\"obius inversion formula for factorizations of zeta functions as a product of $L$-functions, to the $E_1$-page of an equivariant spectral sequence converging to equivariant algebraic $\K$-groups. Additionally, the spectral Mackey functor structure on equivariant $\K$-theory allows us to incorporate certain ramified extensions that appear in these $L$-functions.  
	\end{abstract}
	\maketitle 
	
	\tableofcontents
	
	\section{Introduction}
	The study of the Riemann $\zeta$-function has been a driving force in mathematics since its invention. While the location of its zeros is the subject of the celebrated Riemann Hypothesis, there is another aspect of its study that has garnered much attention in the previous few decades: its special values at non-positive integers,  and connections with algebraic $\K$-theory. The beginning of this connection is actually much older. If $F$ is a number field and $\zeta_F(s)$ is the Dedekind zeta function of $F$ (see \Cref{ex:ddk}), then Dirichlet's class number formula (from early 1800's) can be re-interpreted as
	\[
	\zeta_F^*(0) = -\frac{\# \K_0(\mathcal{O}_F)_{\mathrm{torsion}}}{\# \K_1(\mathcal{O}_F)_{\mathrm{torsion}}}\cdot R_F,
	\]
	where $\zeta_F^*(0)$ is the leading coefficient of the Taylor expansion of $\zeta_F(s)$ around $s=0$, $\Ocal_F$ is the ring of integers of $F$, and $R_F$ is the Dirichlet regulator of $F$. On the other hand, the Birch-Tate conjecture can be reinterpreted in $\K$-theoretic terms as well. It states that for a totally real number field $F$ with $r_1(F)$ real embeddings, we have an equality:
	\[
	\zeta_F^*(-1) = \zeta_F(-1) =(-2)^{r_1(F)}\cdot \frac{\# \K_2(\mathcal{O}_F)}{\# \K_3(\mathcal{O}_F)}.
	\]
	
	In the 70's, Quillen defined higher algebraic $\K$-theory in  \cite{Quillen_higher_alg_K}  and computed $\K_*(\Fq)$ for a finite field $\Fq$ in \cite{quillen-plus}. Around the same time, Lichtenbaum studied \'etale cohomology groups of number fields in \cite{Lichtenbaum_1972}. These computations led to the \textbf{Quillen--Lichtenbaum Conjecture} (QLC), first proposed in \cite{Lichtenbaum_1973}, describing  $\zeta_F^*(1-k)$ in terms of the orders of  certain $\K$-groups, building a very compelling bridge between the seemingly disparate fields of number theory and homotopy theory. 
	
	To be slightly more precise, the QLC proposes that special values of the Dedekind $\zeta$-function $\zeta_F$ for a number field $F$ are computed by the higher algebraic $\K$-groups of its ring of integers $\Ocal_F$ in the sense that the following equality holds up to a power of $2$:
	\begin{equation*}
		\zeta^*_{F}(1-k)=\pm\frac{\# \K_{2k-2}(\mathcal{O}_F)}{\# \K_{2k-1}(\mathcal{O}_F)_{\mathrm{torsion}}}\cdot R^B_k(F), \qquad k \geq 2.
	\end{equation*}
	Here, $R^B_k(F)$ is the $k$-th Borel/Beilinson regulator.\footnote{The two regulators only differ by powers of $2$ by \cite[Theorem 10.9]{BurgosGil_reg}.} We also note that the Dirichlet's Unit Theorem and Borel's calculation of rational algebraic $\K$-groups of number fields imply that  
	\begin{equation*}
		\ord_{s=1-k} \zeta_F(s)=\mathrm{rank}(\K_{2k-1}(\Ocal_F)), \qquad k\ge 1.
	\end{equation*}
	We refer the readers to \Cref{sec:ql-ff} for a review on the Quillen--Lichtenbaum Conjectures for zeta functions of finite, function, and number fields. 
	
	The QLC has been settled in large part due Rost--Voevodsky's proof of the Bloch--Kato conjecture  \cite{voevodsky-z2,voevodsky-zl}. It was a remarkable achievement. Among other things, it led to the development of Morel and Voevodsky's $\mathbb{A}^1$-homotopy theory \cite{A1-homotopy-theory}, an environment in which one can do homotopy theory within algebraic geometry. This subject has, nowadays, taken a life of its own.
	\subsection{What is done in this paper} In this paper, we propose and prove in some cases, a \emph{refinement} of the QLC.  Let us discuss the simplest possible setup --- the context of finite fields. In \cite{quillen-plus}, Quillen computed algebraic $\K$-groups of finite fields as:
	\begin{equation*}
		\K_{t}(\Fq)=\begin{cases}
			\Z, & t=0;\\
			\Z/(q^k-1),& t=2k-1>0 ;\\
			0,& \text{else}. 
		\end{cases}
	\end{equation*}
	This is the first complete computation of $\K$-groups of a class of fields, yielding the simplest case of a QLC type equality which plainly states:
	\begin{equation*}
		\zeta(\Fq,-k)=\frac{1}{1-q^k}=-\frac{\# \K_{2k}(\Fq)}{\# \K_{2k-1}(\Fq)}, \qquad k\ge 1.
	\end{equation*}
	Now, fix $m \ge 1$ and let  $\Fr_q$ be the Frobenius element (generator) of the Galois group $ \gal(\Fbb_{q^m}/\Fq)\cong C_m$. Consider a character $\chi\colon \gal(\Fbb_{q^m}/\Fq)\to \Cx$. In this context, we can also attach an $L$-function to $\chi$ given by a simple ``twist'' of the zeta function of $\Fq$:
	\begin{equation*}
		L(\Fq,\chi,s)=\frac{1}{1-\chi(\Fr_q) q^{-s}}.
	\end{equation*} 
	\begin{theorem}[Special case of \Cref{thm:eQLC_finfld} and \Cref{cor:size-ff-1d}]\label{thm:main_Fq2}
		Let $\sigma$ be the sign character of the group $\gal\left(\Fbb_{q^2}/\Fq\right)=C_2$ and  $\mathrm{S}^{1-\sigma}$ be the representation sphere for the virtual $C_2$-representation $1-\sigma$. Then
		\begin{equation*}
			\pi_{t}^{C_2}(\K(\Fbb_{q^2})\otimes \mathrm{S}^{1-\sigma})=\begin{cases}
				\Z/(q^k+1), & t=2k-1>0;\\
				0, & \text{else.}
			\end{cases}
		\end{equation*}
		It follows that:
		\begin{equation*}
			\frac{1}{1+q^k}=L(\Fq,\sigma,-k)=\frac{\# \pi_{2k}^{C_2}\left(\K(\Fbb_{q^2})\otimes \mathrm{S}^{1-\sigma}\right)}{\# \pi_{2k-1}^{C_2}\left( \K(\Fbb_{q^2})\otimes \mathrm{S}^{1-\sigma}\right)}, \quad k \geq 1.
		\end{equation*} 
	\end{theorem}
	We note the appearance  the equivariant homotopy groups on the $\K$-theory sides of this equality; this is one of the novel contributions of the present paper.  Readers unfamiliar with genuine equivariant homotopy theory are referred to \Cref{subsec:spectral-mackey} for a (hopefully, gentle) introduction to the subject.  
	
	The main theorem of the current paper produces many other cases of the identity in \Cref{thm:main_Fq2} for \textbf{Artin $L$-functions} $L(X,\rho,s)$  (\Cref{def:Artin-L}) of Galois representations $\rho$ of finite, function, and totally real number fields.  Algebraic $\K$-groups are notoriously difficult to compute and, even more so,  equivariant algebraic $\K$-groups. We see our paper as, to our knowledge, the first collection of computations of certain equivariant algebraic $\K$-groups --- at least the ratio of their sizes. 
	
	To state our generalization of the QLC to Artin $L$-functions, we need to first construct spectral liftings of Galois representations. The representation sphere $\mathrm{S}^{1-\sigma}$ in \Cref{thm:main_Fq2} is an integral equivariant Moore spectrum for the sign representation $\sigma$, in the sense that $C_2$ acts on its zeroth (integral) homology group by $\sigma$. In nice cases, we can construct an integral equivariant Moore spectrum $\M(\rho)$ for a Galois representation $\rho$, following the second author's thesis \cite[\S3.3]{nz_Dirichlet_J}.
	\begin{theorem}[\Cref{con:eMoore} and  \Cref{thm:eMoore_higherdim}]\label{thm:eMoore}
		Let $G$ be a finite group and $\rho\colon G\to \GL_N(\Ocal_E)$ be an $E$-linear representation for some number field $E$. If $\rho$ is conjugate to a direct sum of inductions of abelian characters on subgroups, then there exists a connective $G$-CW spectrum $\M(\rho)$ whose integral homology groups are concentrated in degree $0$ and the induced $G$-action  on $\H_0(\M(\rho);\Z)= \Ocal_E^{\oplus N}$ is isomorphic to $\rho$.
	\end{theorem}
	\begin{theorem}[Main Theorem, Theorems \ref{thm:eQLC_finfld}, \ref{thm:eQLC_numfld}, and \ref{prop:eQLC_reduction}]\label{thm:main}
		Let the representation $\rho\colon G=\aut_X(Y)\to \GL_N(\Ocal_E)$ and its equivariant Moore spectrum $\M (\rho)$ be as described in \Cref{thm:eMoore}. We have the following QLC type identities as in \Cref{thm:main_Fq2}:
		\begin{enumerate}
			\item (Finite fields)  $Y=\spec \F_{q^m}$ and $X=\spec \Fq$ are finite fields with Galois group $G=C_m$. Then:
			\begin{equation*}
				\norm_{E/\Q} L(\Fq,\rho,-k)=(-1)^{\dim_\Q (\rho\otimes\Q)^G}\cdot \frac{\#\pi_{2k}^{C_m}\left(\K(\F_{q^m})\otimes \M (\rho)\right)}{\#\pi_{2k-1}^{C_m}\left(\K(\F_{q^m})\otimes \M (\rho)\right)},\qquad k\ge 1.
			\end{equation*}
			In particular, when $\rho=\chi$ is a $1$-dimensional representation, we have: (\Cref{cor:size-ff-1d}) 
			\begin{equation*}
				\pi_{*}^{C_m}(\K(\Fbb_{q^m})\otimes \M (\chi))=\begin{cases}
					\Z/\mathrm{Norm}_{\Q[\image \chi]/\Q} L(\Fq,\chi,-k)^{-1}\cong \Z[\image \chi]/ L(\Fq,\chi,-k)^{-1}, & *=2k-1>0;\\
					0, &\text{else.}
				\end{cases}
			\end{equation*}
			\item (Function fields) $Y\to X$ is a finite morphism of integral, normal, and affine curves over a finite field $\Fq$ such that $\aut_X(Y)= G$ and the induced map at the generic points  (function fields) is a $G$-Galois extension. Then:
			\begin{equation*}
				\norm_{E/\Q}	L(X,\rho,1-k)		=(-1)^{\dim_\Q[(\K_1(Y)\oplus\Z_{\mathrm{triv}})\otimes \rho\otimes \Q]^G}\cdot\frac{\#\pi_{2k-2}^{G}\left(\K(Y)\otimes \M (\rho)\right)}{\#\pi_{2k-1}^{G}\left(\K(Y)\otimes \M (\rho)\right)}, \qquad k\ge 2.
			\end{equation*}
			\item (Totally real number fields) $Y=\spec \Ocal_{F'}$ and $X=\spec\Ocal_F$ where $F'/F$ is a $G$-Galois extension of totally real and abelian number fields. Then:
			\begin{equation*}
				\norm_{E/\Q}	L(\Ocal_F,\rho,1-2n)= ((-1)^n\cdot 2)^{r_1(F)\cdot (\dim_\Q \rho\otimes \Q)}\cdot \frac{\#\pi_{4n-2}^{G}\left(\K(\Ocal_{F'})\otimes \M (\rho)\right)}{\#\pi_{4n-1}^{G}\left(\K(\Ocal_{F'})\otimes \M (\rho)\right)},\qquad n\ge 1,
			\end{equation*}
			where  $r_1(F)$ is the number of real embeddings of $F$. 
			
			The equality only holds up to possible powers of $2$ if $F'$ is only totally real but not abelian over $\Q$. 
		\end{enumerate}
	\end{theorem}
	We make some remarks about this result:
	
	\begin{remark}[Equivariant $\K$-groups] 
		In the case of finite fields, we see our computation as an equivariant extension of Quillen's celebrated computation of algebraic $\K$-groups of finite fields in \cite{quillen-plus}.
	\end{remark}
	
	\begin{remark}[Use of genuine equivariant homotopy theory] Naively, one might want to express the $\K$-theory side in terms of homotopy fixed points of the $G$-action on $\K(Y)$. However, this formulation does not quite work because the map $Y \rightarrow X$ might have ramification points, i.e., it is not a Galois or \'etale cover. To overcome this, we use genuine equivariant homotopy theory via the language of spectral Mackey functors as in \Cref{subsec:spectral-mackey}. This structure helps us record the ramification at every intermediate extension which is strictly more data than the $G$-action on $\K(Y)$. 
		
		Notably, algebraic $\K$-theory of a Galois cover of schemes is known to have a spectral $G$-Mackey functor/genuine $G$-spectrum structure by Barwick \cite{Barwick_spectral_Mackey_I} and Merling \cite{Merling_equivar_alg_k}. We  extend this result to the \emph{pseudo} Galois covers (in the sense of the function and number field cases of \Cref{thm:main}; see also \Cref{def:p-gal}) of $1$-dimensional schemes in \Cref{constr:span-k}.  
	\end{remark}

	\begin{remark}[Integral results] Inverting the order of the group erases the difference between equivariant homotopy groups and homotopy groups of the homotopy fixed points. A simpler version of \Cref{thm:main}, using just fixed points on $\K$-groups, is easier to prove after inverting the order of the group. Hence the main contribution of this paper is to provide an integral statement without inverting any  integer on $\K$-theory, except for possibly the prime $2$ in the totally real but non-abelian number field case.  
	\end{remark}

	\begin{remark}[Equivariant Moore spectra] One of the key characters of the second author's thesis \cite{nz_Dirichlet_J} is the equivariant Moore spectrum. For us, the equivariant cellular structure on this object is crucial to the proof of \Cref{thm:main}. Essentially, it ``animates'' the M\"obius inversion formula. 
		
		Let us recall that in classical topology, the Moore spectrum associated to an abelian group $A$ is a spectrum $\M A$ characterized as the unique connective spectrum such that $\H_0(\M A;\Z) \cong A$ and $\H_{>0}(\M A; \mathbb{Z}) = 0$. In addition,  homotopy groups of a spectrum $X$ with coefficients in an abelian group $A$ are simply  homotopy groups of $X \otimes \M A$. In particular, algebraic $\K$-groups of a ring $R$ with coefficients in $A$ are then  homotopy groups of the spectrum $\K(R) \otimes \M A$.\footnote{This is different from $\K_*(R)\otimes_\Z A$. The two are connected by a (non-split) Universal Coefficient Theorem.}
		
		However, we note that the existence and uniqueness of equivariant Moore spectra are subtle questions. If $G$ is a group acting on $A$, then Steenrod asked if one can always promote the action to one on the Moore spectrum $\M A$, so that the induced action on $\H_0$ is the prescribed one. This has been answered in the negative by Carlsson \cite{carlsson-cex}. Even if the $G$-action can be promoted to $\M A$, the lifting is not unique in general. See discussions in  \Cref{rem:eMoore_uniqueness}. 
	\end{remark}
	Applying our method to rational equivariant algebraic $\K$-theory of number fields, we recover a generalization of Borel's rank \Cref{thm:Borel_rank} to Artin $L$-functions by Gross in \cite{Gross_Artin_L}, which was reformulated by the second author in \cite{nz_QB_AL}.
	\begin{theorem}[Gross, Zhang, {\Cref{thm:twisted_Borel-1d} and \Cref{prop:twisted_Borel}}]
		Let $F'/F$ be a Galois extension of number fields and $\rho\colon G\to\GL_N(\Ocal_E)$ be an $E$-linear Galois representation for some number field $E$. Then there exists a unique rational equivariant Moore spectrum $M\Q(\rho)$ such that 
		\begin{equation*}
			\ord_{s=1-n}L(\Ocal_F,\rho,s)=\dim_E \pi^G_{2n-1}(\K(\Ocal_{F'})\otimes M\Q(\rho)).
		\end{equation*}
	\end{theorem}
	\subsection{Outline of proof} 
	We break the proof of \Cref{thm:main} into two major steps: the first, more difficult step, is a calculation of the equivariant $\K$-theory with coefficients in a primitive abelian character; the second step is a reduction to the case of primitive characters of a cyclic groups. We explain these steps in turn. 
	\subsubsection{Step I: Calculation in the case of primitive character of a cyclic group} We begin with some purely number theoretic reductions. In the situation when $G=C_m$ is cyclic, $Y/X$ is a pseudo $C_m$-Galois cover of schemes of dimension $d\le 1$, and $\rho = \chi\colon C_m \hookrightarrow \mathbb{C}^{\times}$ is a primitive abelian character. We exploit the fact that when $G$ is cyclic, conjugacy classes of characters form a partially ordered set (by their kernels) and thus we can apply the \textbf{M\"obius Inversion Formula} to the factorization formula in \Cref{subsec:inv_fact}:
	\begin{align*}
		\zeta(Y,s)=\prod_{\chi\in C_m^\vee} L(X,\chi,s) &&\Longrightarrow &	\prod_{\chi\colon C_m\hookrightarrow \Cx}L(X,\chi,s)
		= \prod_{j=0}^\lambda\left(\prod_{1\le k_1<\cdots<k_j\le \lambda} \zeta\left(Y/_{C_{p_{k_1}\cdots p_{k_j}}},s\right)\right)^{(-1)^j},
	\end{align*}
	where $\{p_1,\cdots,p_\lambda\}$ is the set of prime factors of $m$ and $Y/_H$ is the quotient scheme of $Y$ under the action of a subgroup $H\le C_m$. 
	
	In the cases of finite, function, and totally real number fields, the special values of $L$-functions at negative integers are algebraic numbers in $\Q[\zeta_m]$. Furthermore, the product of $L(X,\chi,d-n)$ over injective characters is equal to the norm of this algebraic number over $\Q$. The classical QLC (see \Cref{sec:ql-ff}) states that the values of those zeta functions at negative integers are ratios between sizes of algebraic $\K$-groups. In this way, we can express norms of special values of $L$-functions as alternating products of sizes of algebraic $\K$-groups:
	\begin{equation}\label{eqn:norm_alt_prod_K}
		|\norm_{\Q[\zeta_m]/\Q}L(X,\chi,d-n)| = \frac{\prod_{j=0}^\lambda\left(\prod_{1\le k_1<\cdots<k_j\le \lambda} \# \K_{2n-2d}\left(Y/_{C_{p_{k_1}\cdots p_{k_j}}}\right)\right)^{(-1)^j}}{\prod_{j=0}^\lambda\left(\prod_{1\le k_1<\cdots<k_j\le \lambda} \# \K_{2n-1}\left(Y/_{C_{p_{k_1}\cdots p_{k_j}}}\right)\right)^{(-1)^j}}.
	\end{equation}
	
	Let us switch gear now and set out to compute the equivariant $\K$-groups with coefficients in a Moore spectrum. Let $M$ be a dualizable $G$-spectrum. We can try to calculate the equivariant homotopy groups of $\K(Y)$ with coefficients in $M$ via the \textbf{equivariant Atiyah-Hirzebruch spectral sequence} (eAHSS):
	\begin{equation}\label{eq:eahss}
		E_2^{s,t} = \H^s_G(M; \underline{\pi}_t(\K(Y))) \Longrightarrow \pi_{t-s}^G(\K(Y) \otimes \Dbb M),
	\end{equation}
	where $\Dbb M:=\underline{\map}(M,\mathrm{S}^0)$ is the Spanier-Whitehead dual of $M$ in $G$-spectra. The $E_2$-page of this spectral sequence is Bredon cohomology, the $G$-equivariant analog of singular cohomology. These groups are difficult to compute in general, but one can try to get a handle on them if there is a good equivariant cellular decomposition of $M$. The latter gives an $E_1$-refinement to the $E_2$ page of the spectral sequence \eqref{eq:eahss}. When $\Dbb M=\M(\chi)$ is the equivariant Moore spectrum for the character $\chi$,  we have the cellular structure discussed in \Cref{constr:cell_ss}, whence an equivariant filtration on the $G$-spectrum $\K(Y) \otimes \M(\chi)$. This cellular structure comes equipped with an explicit description of its graded pieces, described in terms of fixed points with respect to various subgroups of a cyclic group $G$.  The $E_1$-page of the eAHSS \eqref{eq:eahss} is then:
	\begin{equation}\label{eqn:eAHSS_E1_intro}
		E_1^{s,t}=\bigoplus_{1\le k_1<\cdots<k_s\le \lambda} \K_{t}\left(Y/_{C_{p_{k_1}\cdots p_{k_s}}}\right)\Longrightarrow \pi_{s+t}^{C_m}(\K(Y) \otimes \M(\chi)). 
	\end{equation}
	
	To compute this spectral sequence, one then observes two further phenomena: that the spectral sequence \eqref{eq:eahss} collapses at the $E_2$-page and that the passage from the $E_1$-page to the $E_2$-page is not too ``lossy''. This amounts to a sparsity result on the level of Bredon cohomology for these Moore spectra; see Propositions~\ref{prop:Bredon_psi_m_fixedpt} and~\ref{prop:Bredon_psi_m_general}. There is a $\K$-theoretic input in these calculations; specifically for the former case: without inverting the order of $G$ odd $\K$-groups of certain extensions satisfy descent on the level of homotopy groups (Corollaries~\ref{cor:finite_fixed},~\ref{cor:fixed-curve} and~\ref{cor:fixed-numfld_odd}). 
	
	Finally, the $L$-function and the equivariant homotopy theory sides of the proof are connected by the remarkable observation that the alternating products of sizes of algebraic $\K$-groups in \eqref{eqn:norm_alt_prod_K} are precisely the \textbf{Euler numbers} of the equivariant cellular chain complexes on the $E_1$-page of the eAHSS in \eqref{eqn:eAHSS_E1_intro}. The ratios of these Euler numbers are invariant when passing to the $E_2$-page (taking cohomology), and then to the $E_\infty$-page since the $E_2$-page has only two non-zero rows and the eAHSS collapses there. 
	
	Everything above is summarized below in the case of primitive characters of pseudo $C_m$-Galois covers of affine curves (the signs are ignored for simplicity): 
	\begin{align*}
		&\quad~|\norm_{\Q[\zeta_m]/\Q}L(X,\chi,1-n)| &&\\
		&= \prod_{j=0}^\lambda\left(\prod_{1\le k_1<\cdots<k_j\le \lambda} \left|\zeta\left(Y/_{C_{p_{k_1}\cdots p_{k_j}}},1-n\right)\right|\right)^{(-1)^j} && \textup{M\"obius Inversion Formulae~\Cref{subsec:inv_fact}}\\
		&= \prod_{j=0}^\lambda\left(\prod_{1\le k_1<\cdots<k_j\le \lambda}\frac{\# \K_{2n-2}\left(Y/_{C_{p_{k_1}\cdots p_{k_j}}}\right)}{\# \K_{2n-1}\left(Y/_{C_{p_{k_1}\cdots p_{k_j}}}\right)}\right)^{(-1)^j}&& \textup{Quillen--Lichtenbaum Conjectures \Cref{sec:ql-ff}}\\
		&= \frac{\prod_{j=0}^\lambda\left(\prod_{1\le k_1<\cdots<k_j\le \lambda} \# \K_{2n-2}\left(Y/_{C_{p_{k_1}\cdots p_{k_j}}}\right)\right)^{(-1)^j}}{\prod_{j=0}^\lambda\left(\prod_{1\le k_1<\cdots<k_j\le \lambda} \# \K_{2n-1}\left(Y/_{C_{p_{k_1}\cdots p_{k_j}}}\right)\right)^{(-1)^j}}&&\\
		&=\frac{\text{Euler number}(E_1^{*,2n-2}, d_1)}{\text{Euler number}(E_1^{*,2n-1},d_1)}&& \textup{eAHSS \eqref{eqn:eAHSS_E1_intro}}\\
		&=\frac{\text{Euler number}(E_2^{*,2n-2})}{\text{Euler number}(E_2^{*,2n-1})}&& \text{Euler numbers invariant under cohomology}\\
		&=\frac{\# \H_{C_m}^0(\M (\psi_m);\underline{\K}_{2n-2}(Y))/ \# \H_{C_m}^{-1}(\M (\psi_m);\underline{\K}_{2n-2}(Y))}{\# \H_{C_m}^0(\M (\psi_m);\underline{\K}_{2n-1}(Y))}&& \textup{Propositions \ref{prop:Bredon_psi_m_fixedpt} and \ref{prop:Bredon_psi_m_general}}\\
		&=\frac{\# E_2^{0,2n-2}}{\# E_2^{0,2n-1}\cdot\# E_2^{-1,2n-2}}&& \\
		&=   \frac{\# \pi_{2n-2}^{C_m}\left(\K(Y)\otimes \M (\chi)\right)}{\# \pi_{2n-1}^{C_m}\left(\K(Y)\otimes \M (\chi)\right)}. && \textup{eAHSS collapses on $E_2$-page}
	\end{align*}	
	\subsubsection{Step II: Reduction to the primitive characters of cyclic groups} We notice some similarities between identities of Artin $L$-functions and isomorphisms of equivariant homotopy groups as summarized in \Cref{table:analogies}. We think of these identities as a sequence of moves that one can make to reduce $G$ to a cyclic group and the representation to an abelian character; the point is that these moves can be both be made in the $L$-function and the $\K$-theory side.  In addition, $L$-functions satisfy a descent identity ((3) in \Cref{prop:Artin-L_properties}) that is similar to the descent property in equivariant homotopy theory (see \Cref{prop:desc}), which further reduces the case to \emph{primitive} abelian characters. 
	
	\begin{table}[ht]
		\renewcommand{\arraystretch}{1.5}
		\caption{Comparison of identities and isomorphisms}\label{table:analogies}
		\begin{tabu}{|c|c|}
			\hline $L$-functions &Equivariant homotopy groups\\\hline
			$L(X,1_G,s)=\zeta(X)$& $\pi^G_*(M)\cong \pi_*(M^G)$ \\\hline
			$L(X,\rho_1\oplus \rho_2,s)=L(X,\rho_1,s)\cdot L(X,\rho_2,s)$& $\pi_*^G(M_1\oplus M_2)=\pi_*^G(M_1)\oplus \pi_*^G(M_2)$\\\hline
			$L(X,\mathrm{Ind}_H^G\rho,s)=L(Y/_H,\rho,s)$ &$\pi^G_*(\mathrm{Ind}_H^GM)\cong \pi_*^H(M)$\\\hline
		\end{tabu}          
	\end{table}
	
	More precisely, if we have a representation of the form $\rho=\bigoplus_{i} \mathrm{Ind}_{H_i}^G\chi_i$, which is conjugate to a sum of inductions of abelian characters $\chi_i$ on subgroups $H_i\le G$, then we can construct a lift of it to the sphere spectrum by setting $ \M (\rho)=\bigoplus_{i} \mathrm{Ind}_{H_i}^G\M (\chi_i)^{\oplus [E:\Q[\image \chi_i]]}$; that this is a Moore spectrum for $\rho$ is verified in \Cref{thm:eMoore_higherdim}. We then prove in \Cref{prop:eQLC_reduction} that the Main \Cref{thm:main} for $\rho$ reduces to the primitive abelian character $\chi$ case by the comparison table above.
	\subsection{Related work} To our knowledge, the first instance of a $\K$-theoretic interpretation of Artin $L$-functions was spelled out in the work of Kolster \cite{kolster} in which he considers the sign representation of a CM extension of number fields. His methods proceed via \'etale cohomology and hence, just like we do, use the Quillen--Lichtenbaum Conjectures. As the odd $\K$-groups of number fields with complex embeddings have non-trivial free parts, the result of this paper does not cover his. It would be interesting to cover his situation as well by taking into account of regulators (see \Cref{sec:future} below).
	
	As explained in \cite[Section 1.4]{Huber-Kings_2003}, Bloch--Kato's conjecture \cite{bloch-kato} concerning cohomological interpretations of $L$-function of Artin motives implies Lichtenbaum's conjecture (relating \'etale cohomology and special values of zeta functions) when the Artin motive is $h^0(F)$ for $F$ a field. Furthermore, in \cite{Huber-Kings_2003}, the Bloch--Kato conjecture was verified for Artin motives associated to Dirichlet characters. Therefore, at least in spirit, our results are related to the work of \cite{Huber-Kings_2003}. We leave it to interested readers to work out an explicit relationship. 
	
	\subsection{Future work}\label{sec:future}
	In future work, we hope to remove the assumptions in \Cref{thm:eMoore} and \Cref{thm:main}. Two major ones are:
	\begin{enumerate}
		\item The Galois representations are assumed to be a sum of inductions of abelian characters on subgroups in \Cref{thm:eMoore}. 
		\item In the number field case, we can only prove \Cref{thm:main} when $n$ is even and the number fields are totally real. 
	\end{enumerate}
	
	The first assumption is related to the \textbf{Brauer Induction Theorem} (see \cite[Theorems 10.20]{Serre_rep1977}), which says that  the complex representation ring $R(G)$ of a finite group $G$ is generated by abelian characters up to inductions.  This only implies any complex representation is a \emph{virtual sum} of inductions of abelian characters, not necessarily a \emph{direct} sum.  See more discussions in \Cref{rem:Brauer_ind}. In future works, we hope to construct integral equivariant Moore spectra for general Galois representations of finite groups as $G$-CW spectra. This will help us prove \Cref{thm:main} for general higher-dimensional Galois representations. 
	
	The second assumption is related to \textbf{regulators} of $\zeta$-functions and free parts of algebraic $\K$-groups of number fields. On the $\K$-theory side, our proof relies crucially on the finiteness of $\K_{4n-1}(\Ocal_F)$ when $F$ is totally real. In this case, the equivariant cellular chain complex to compute the Bredon cohomology $\H^*_G(\Dbb \M (\chi);\underline{\K}_{4n-1}(\Ocal_F))$ is a bounded complex of finite abelian groups. Hence the size of $\H^0_G$ is equal to the Euler number of the complex if it is the only non-zero cohomology group. On the $\zeta$-function side, $\zeta_F(1-k)=0$ unless $F$ is totally real and $k$ is even. If either fails, the QLC states that
	\begin{equation*}
		\zeta^*_{F}(1-k)=\pm \frac{\# \K_{2k-2}(\mathcal{O}_F)}{\# \K_{2k-1}(\mathcal{O}_F)_{\mathrm{torsion}}}\cdot R^B_k(F),
	\end{equation*}
	holds up to powers of $2$, where $\zeta^*_F(1-k)$ is the leading coefficient in the Taylor expansion of $\zeta_F(s)$ at $s=1-k$, and $R^B_k(F)$ is the  $k$-th Beilinson/Borel regulator of $F$. In future works, we plan to prove \Cref{thm:main} for general number fields, using corresponding non-equivariant Quillen--Lichtenbaum Conjecture for those number fields (see \Cref{thm:QLC_numfld}). We thus end this introduction with what we believe is a reasonable generalized version of QLC.
	\begin{conjecture}[Generalized QLC]\label{conj:gen-qlc} Let $F'/F$  be a $G$-Galois extension of number fields for a finite group $G$ and $\rho\colon G\to \GL_N(\Ocal_E)$ be an $E$-linear Galois representation for a number field $E$. Then, up to a powers of $2$, we have an equality
		\begin{equation*}
			\prod_{\tau\in \gal(E/\Q)} L^*(\Ocal_F,\tau\circ \rho,1-k)= \pm \frac{\# \pi^G_{2k-2}(\K(\Ocal_{F'})\otimes \M (\rho))}{\#\pi^G_{2k-1}(\K(\Ocal_{F'})\otimes \M (\rho))_\mathrm{torsion}}\cdot \prod_{\tau\in \gal(E/\Q)} R_{k,\tau\circ \rho} \qquad k \geq 1,
		\end{equation*}
		where $L^*(\Ocal_F,\rho,1-k)$ is the leading coefficient of the Taylor series of the $L$-function at $s=1-k$ and $R_{k,\rho}$ is the $k$-th regulator of $\rho$.
	\end{conjecture}
	\subsection*{Notation and convention} We freely use the language and syntax of higher algebra as laid out in \cite{lurie-ha}; in particular the foundations of equivariant homotopy theory that we use is expressed in terms of $\infty$-categories as in the work of Barwick \cite{Barwick_spectral_Mackey_I}.
	\begin{itemize} 
		\item (Euler numbers) Suppose that we have a bounded chain complex of \emph{finite} abelian groups
		\[
		D_{*} := 0 \rightarrow D_m \rightarrow D_{m-1} \rightarrow \cdots\rightarrow D_{m-k} \rightarrow 0.
		\]
		Then we set
		\[
		\text{Euler number}(D_{*}) := \prod \#\left( D_j \right)^{(-1)^j}. 
		\]
		Note that this number is invariant under taking (co)homology. 
		\item (Spanier-Whitehead duals) If $(\mathcal{C}, \otimes, \1)$ is a presentably symmetric monoidal $\infty$-category, then we may speak of a dualizable object $X$ and their categorical/Spanier-Whitehead dual $\mathbb{D}(X)$ which is calculated as the internal hom $\underline{\map}(X, \1)$.
		\item (Representations) For a group $G$, we write $1_G$ and $\rho_G$ for its complex trivial and regular representations, respectively.
		\item (primitive/injective characters) In this paper, a \textbf{primitive} character $\chi\colon G\to \Cx$ is simply an \textbf{injective} group homomorphism. So we will use the terms primitive/injective characters interchangeably. Such representations are also called \textbf{faithful} in representation theory. We remind the readers that this is different from the notion of a \emph{primitive Dirichlet character}, where $\chi\colon \znx\to \Cx$ is called primitive if it cannot factor through $(\Z/N')^\times$ for any proper divisor $N'$ of $N$ (see \cite[\S VII.2]{Neukirch_ANT}). 
		\item The Euler totient function $\phi(m)$ of a positive integer $m$ is the number of positive integers $1\le k\le m$ that are coprime to $m$. 
		\item We will fix the meanings of the following letters (in standard math font) throughout the paper:
		\begin{itemize}
			\item $d$ is the (Krull) dimension of a scheme $X$.
			\item $\kappa(x)$ is the residue field of a closed point $x$ in a scheme $X$. 
			\item $k(X)$ is the function field of an integral scheme $X$. 
			\item $q$ is the size of a finite field $\Fq$.
			\item $\ell$ is a prime number not dividing the size of the base field $\Fq$. 
			\item $F$ and $F'$ are number fields whose zeta functions and algebraic $\K$-theory are studied.
			\item $r_1(F)$ is the number of real embeddings of a  number field $F$.
			\item $E$ is a number field such that a Galois representation $\rho$ is $\Ocal_E$-linear. 
			\item$\norm_{E/\Q}(x)$ is the field norm of an algebraic number $x$ in $E$ over $\Q$.
			\item $N$ is the dimension of a representation $\rho$ over $E$. 
			\item $G$ is a finite group.
			\item $C_m$ is the cyclic group of order $m$.
			\item $\psi_m\colon C_m\hookrightarrow \Z[\zeta_m]^\times$ is a fixed integral cyclotomic character of $C_m$.
			\item $\sigma$ is the sign character/representation of $C_2$.
			\item $\lambda$ is the number of distinct prime factors of an integer $m$.
		\end{itemize}
	\end{itemize}
	\subsection*{Acknowledgments} This project has benefitted from numerous conversations that the authors had with Matthew Ando, Paul Balmer, Clark Barwick, Mark Behrens, Ted Chinburg, Paul Goerss, Lars Hesselholt, Mike Hopkins, Annette Huber-Klawitter, Michael Larsen, Yifeng Liu,  Mona Merling, Matthew Morrow, Niranjan Ramachandran, Catherine Ray, Charles Rezk, Nick Rozenblyum, Peter Scholze, and Zhouli Xu. In particular, we would like to thank Tobias Barthel and Michael Mandell for a careful reading of the draft of the paper.
	
	EE has benefitted from an Erik Ellentuck fellowship from the Institute for Advanced Study and expresses gratitude for perfect working conditions.
	
	NZ was partially supported by the NSF Grant DMS-2348963/2304719 and the AMS-Simons Travel Grant.  Some of the work was done when NZ was visiting the Max Planck Institute for Mathematics.  He thanks the MPIM for their hospitality and support.
	\section{Zeta and $L$-functions} \label{sec:zeta-l} We begin with an exposition of some concepts in number theory and arithmetic algebraic geometry. We hope that the exposition will be friendly to homotopy theorists who are not necessarily well-acquainted with the apparatus of zeta and $L$-functions. We will follow Serre's exposition in \cite{Serre_zeta_L}. 
	\subsection{Zeta functions}
	To begin with, the most fundamental function in the subject is arguably the following:
	
	\begin{definition}\label{def:hs-weil} Let $X$ be a finite type $\mathbb{Z}$-scheme of which is relatively equidimensional over $\mathbb{Z}$ of relative dimension $d$. We define the \textbf{Hasse--Weil zeta function} attached to $X$ to be
		\begin{equation*}
			\zeta(X,s)= \prod_{x \in |X|} \frac{1}{1- \#\kappa(x)^{-s}};
		\end{equation*} 
		where $|X|$ is the topological space (set) of closed points on the scheme $X$ and $\#k$ denotes the cardinality of a finite field $k$.  Notice that the residue field $\kappa(x)$ of a closed point $x\in |X|$ is also finite, since $X$ is finite type over $\Z$. This product is absolutely convergent over the domain $\mathrm{Re}(s)>d$ in $\Cbb$.
	\end{definition}
	\begin{example}\label{ex:finite-field} If $X = \spec (\Fbb_q)$ is a finite field,  then we have
		\begin{equation*}
			\zeta(\Fq,s)=\frac{1}{1-q^{-s}}.
		\end{equation*}
	\end{example}
	\begin{example}\label{ex:function-field} If the structure map $X \rightarrow \spec(\mathbb{Z})$ factors through $\spec(\mathbb{F}_q) \rightarrow \spec(\mathbb{Z})$ for some finite field $\Fq$, then we can describe the Hasse-Weil zeta function in terms of point count. Namely,  we  have an equality (see proof in \Cref{rem:katz}):
		\begin{equation*}
			\zeta(X,s)=\exp\left(\sum_{m=1}^{\infty}\frac{\# X(\Fbb_{q^m})}{m}\cdot q^{-ms}\right).
		\end{equation*}
		This is often called the \textbf{Weil zeta function} of the $\Fq$-scheme $X$. In particular, when $X=\spec\Fq$, we recover the formula in \Cref{ex:finite-field}. 
	\end{example}	
	\begin{example} \label{ex:ddk} Suppose that $X = \spec(\mathcal{O}_F)$ where $F$ is a number field and $\mathcal{O}_F$ is its ring of integers. Then we recover the \textbf{Dedekind zeta function}, in the sense we get equalities:
		\begin{equation}\label{def:hs-weil_numfld}
			\zeta_F(s):=\zeta(\Ocal_F,s)=\prod_{(0)\ne \mathfrak{p}\trianglelefteq \Ocal_F} \left(1-(\# \Ocal_F/\mathfrak{p})^{-s}\right)^{-1}=\sum_{(0)\ne \Ical\trianglelefteq \Ocal_F} \frac{1}{(\# \Ocal_F/\Ical)^s},
		\end{equation} 
		where the product and the summation range over all non-zero \emph{prime} ideals and all non-zero ideals and of $\Ocal_{F}$, respectively. The last equality boils down to the unique factorization of ideals in the Dedekind domain $\Ocal_F$. When $F = \mathbb{Q}$, we recover the \textbf{Riemann zeta function}:
		\begin{equation*}
			\zeta(s):=\zeta_\Q(s)=\prod_{p\text{ prime}}\frac{1}{1-p^{-s}}=\sum_{n=1}^\infty \frac{1}{n^s}. 
		\end{equation*}
	\end{example}
	
	The next theorem records analytic properties of the Hasse-Weil zeta function in certain situations. 	
	\begin{theorem}
		The $\zeta$-functions in the three examples above admit meromorphic continuations to the complex plane.
	\end{theorem}
	\begin{proof} We give a sketch of proof for each case in turn.
		\begin{enumerate}
			\item When $X=\spec \Fq$, this is obvious since $\zeta(\Fq,s)=\frac{1}{1-q^{-s}}$ is a meromorphic function on $\Cbb$ with a simple pole at $s=0$.
			\item When $X$ is a smooth variety of dimension $d$ over a finite field $\Fq$, the meromorphic continuation follows from the rationality part of the Weil Conjectures by Dwork  in \cite{Dwork_rationality1960}. This can be proved using the Lefschetz fixed point theorem for \'etale cohomology. Namely, we have an identity  \cite[Theorem 27.6]{milneLEC}:
			\begin{equation*}
				\zeta(X,s)=\prod_{r=0}^{2d}\det\left[1-\Fr_q \cdot q^{-s}\mid \H^r_\et(\overline{X};\Qell)\right]^{(-1)^{r+1}},
			\end{equation*}
			where $\overline{X}=X\times_{\Fq} \spec \overline{\F}_q$, $\Fr_q$ is the $q$-th power Frobenius endomorphism on $\overline{X}$, and $\ell$ is a prime number that does not divide $q$. 
			\item When $X=\spec \Ocal_{F}$ for a number field $F$, the summation \eqref{def:hs-weil_numfld} for the Dedekind zeta function $\zeta_F(s)$ converges absolutely and uniformly when $\mathrm{Re} (s)\ge1+\delta$ for every $\delta>0$. One first extends $\zeta_F$ to the strip $0<\mathrm{Re
			}(s)\le 1$. The analytic continuation follows from the functional equation of $\zeta_F$ in this region.  This is a standard result in analytic number theory. See \cite[Corollary VII.5.11.(i)]{Neukirch_ANT}. \qedhere
		\end{enumerate}
	\end{proof}
	\begin{remark}
		The existence of analytic/meromorphic continuation of the Hasse-Weil $\zeta$-functions for finite type varieties over $\Z$ is open in general. Some other known cases include  \emph{cellular} schemes like $\mathbb{A}^d, \mathbb{P}^d$ over $\Ocal_{F}$ (since their zeta functions are products of translations of $\zeta_F(s)$) and elliptic curves (as a consequence of the celebrated modularity theorem). 
	\end{remark}
	\begin{proposition}\label{prop:rational_value}
		In all  three examples above, $\zeta(X,s)$ is a rational number when $s$ is a non-positive integer. 
	\end{proposition}
	\begin{proof}
		Again, the case when $X$ is a smooth variety over a finite field follows from the rationality part of the Weil Conjectures. For $X=\Ocal_{F}$, one can first show $\zeta_F(1-n)=0$ unless $F$ is totally real and $n$ is even. The rationality of $\zeta_F(1-2n)$ when $F$ is totally real  was proved in \cite{Klingen1962}.
	\end{proof}
	\subsection{Artin $L$-functions of Galois representations}\label{subsec:artin} 
	In this subsection, we introduce $L$-functions associated to Galois representations. We work with pseudo-Galois coverings of schemes to incorporate ramifications. 
	\begin{definition}\label{def:p-gal} 
		A finite surjective morphism of noetherian, normal, integral schemes $\pi\colon Y\to X$ is called a \textbf{pseudo-Galois} covering if the induced map on function fields $k(X)\to k(Y)$ is a finite Galois extension and the natural homomorphism $\aut_X(Y)\to \gal(k(Y)/k(X))$ is an isomorphism. A ring map $R\to R'$ is called pseudo-Galois if the induced map $\spec R'\to \spec R$ is a pseudo-Galois covering. If the group $G = \gal(k(Y)/k(X))$ is fixed, we say that $\pi$ is a \textbf{pseudo $G$-Galois cover}. 
	\end{definition}
	
	\begin{remark} The concept in \Cref{def:p-gal} was given the name pseudo-Galois extensions by Suslin and Voevodsky in~\cite[Definition 5.5]{suslin1996singular}.
	\end{remark}
	
	\begin{example}
		\begin{enumerate}
			\item Any Galois covering of normal integral schemes is pseudo-Galois.
			\item Suppose that $X$ is a normal integral scheme with function field $F$ and let $F'$ be a finite, Galois extension of $F$. Let $Y$ be the normalization of $X$ into $F'$ \cite[\href{https://stacks.math.columbia.edu/tag/035H}{Tag 035H}]{stacks-project}, then the map $Y \rightarrow X$ is an example of a pseudo-Galois covering. 
			\item As a more classical version of the above example, we can consider $F'/F$ is a Galois extension of number fields. Then one can check this extension map restricts to their rings of integers $\Ocal_F\to\Ocal_{F'}$, which is a pseudo-Galois extension of Dedekind domains. 
			\item  Geometrically, a pseudo-Galois covering of Riemann surfaces is a ramified covering space. Given a compact Riemann surface $X$, there is a one-to-one correspondence between ramified coverings of $X$ and Galois extensions of the field of meromorphic functions on $X$. See \cite[Theorem 8.12]{Forster_RS}. 
		\end{enumerate}
	\end{example}
	Let $\pi\colon Y\to X$ be a pseudo $G$-Galois covering, $y\in |Y|$ be a closed point and $x=\pi(y)$. We recall the following definitions and results needed to define Artin $L$-functions (see \cite[\href{https://stacks.math.columbia.edu/tag/0BSD}{Tag 0BSD}]{stacks-project}):
	\begin{itemize}
		\item $\Dcal(y)=\{g\in G\mid g(y)=y\}$ is the decomposition group of $y$.
		\item One can show that there are natural surjections:
		\begin{equation*}
			\begin{tikzcd}
				\Dcal(y) \rar[->>]&\gal(\kappa(y)/\kappa(x)).
			\end{tikzcd}
		\end{equation*}
		\item $\Ical(y)$, the inertia subgroup of $y$, is the kernel of the surjection above. The map $\pi$ is unramified at $x$ if $\Ical(y)$ is trivial (for any $y$ above $x$).
		\item $\Fr_{y/x}\in \Dcal(y)/\Ical(y)$ corresponds to the Frobenius element in $\gal(\kappa(y)/\kappa(x))$ under the isomorphism $\gal(\kappa(y)/\kappa(x)) \cong \Dcal(y)/\Ical(y)$.  
		\item If $y'$ is another point above $x$ such that $g(y)=y'$ for some $g\in G$. Then $g^{-1}\Dcal(y')g=\Dcal(y)$, $g^{-1}\Ical(y')g=\Ical(y)$, and $g^{-1}\Fr_{y'/x}g=\Fr_{y/x}$.
		\item We will write $\Fr_x$ and $\Ical(x)$ for the conjugacy classes of the element $\{\Fr_{y/x}\}$ and the subgroups $\{\Ical(y)\}$ for all $y$ above $x$. 
	\end{itemize}
	\begin{definition}\label{def:Artin-L}
		Suppose that $Y \rightarrow X$ is a pseudo-Galois covering of normal integral schemes such that $G =  \gal(k(Y)/k(X))$. Let $\rho\colon G\to \GL_N(\Cbb)$ be a representation of $G$ valued in a finite dimensional complex vector space $V=\Cbb^N$. The \textbf{Artin $L$-function} of $\rho$ is defined to be the Euler product:
		\begin{equation*}
			L(X,\rho,s)=\prod_{x\in |X|}\frac{1}{\det\left[\id-\# \kappa(x)^{-s}\cdot \rho(\Fr_x)\left| V^{\Ical(x)}\right.\right]}.
		\end{equation*}
		This product converges absolutely to a complex analytic function over the domain $\mathrm{Re}(s)>\dim X=\dim Y$. Notice in each Euler factor:
		\begin{itemize}
			\item $\Fr_x$ is represented by an element $\Fr_{y/x}\in \Dcal(y)/\Ical(y)\subseteq G/\Ical(y)$ for some $y$ above $x$. We are allowed to evaluate $\rho$ on $\Fr_{y/x}$ only after restricting the representation to the fixed point subspace $V^{\Ical(y)}$ by the inertia subgroup;
			\item the determinant does not depend on the choice of $y$ over $x$. 
		\end{itemize}
	\end{definition}
	\begin{proposition}[{\cite[Proposition VII.10.4]{Neukirch_ANT}}]\label{prop:Artin-L_properties}
		The Artin $L$-functions satisfy the following identities with respect to operations on representations:
		\begin{enumerate}
			\item  $L(X,1_G,s)=\zeta(X,s).$ where $1_G$ is the $1$-dimensional trivial representation of $G$. 
			\item (Additivity) $L(X,\rho_1,s)\cdot L(X,\rho_2,s)=L(X,\rho_1\oplus \rho_2,s)$.
			\item (Descent) Notice the Galois representation $\rho$ factors uniquely as $\rho\colon G\twoheadrightarrow G/\ker \rho\xrightarrow{\rho '}\GL_N(\Cbb)$. Then
			\begin{equation*}
				L(X,\rho,s)=L(X,\rho',s),
			\end{equation*}
			where the pseudo-Galois covering on the right hand side is $Y/_{\ker \rho}\to X$. Thus it makes sense to drop the total space $Y$ from the notation of an Artin $L$-function.  
			\item (Induction) Let $H$ be a subgroup of $G$. Then for any finite-dimensional complex representation $\rho$ of $H$, we have an induction formula:
			\begin{equation*}
				L(X,\mathrm{Ind}_H^G\rho,s)=L(Y/_H,\rho,s).
			\end{equation*}
		\end{enumerate}
	\end{proposition}
	\begin{corollary}[Factorization of zeta functions, {{\cite[Corollary VII.10.5]{Neukirch_ANT}}}]\label{cor:fac_zeta-to-L}
		For any normal subgroup $H\trianglelefteq G$, we have  
		\begin{equation*}
			\zeta(Y/_H,s)=\prod_{H\subseteq \ker \rho} L(X,\rho,s)^{\dim \rho},
		\end{equation*}
		where $\rho$ ranges through all irreducible complex representations of $G$ whose kernels contain $H$.  If $G$ is abelian, write $G^\vee=\hom(G,\Cx)$ for the group of its complex abelian characters. Then for any subgroup $H\le G$, we have
		\begin{equation*}
			\zeta(Y/_H,s)=\prod_{\chi \in (G/H)^\vee } L(X,\chi,s),
		\end{equation*}
		where complex abelian characters of $G/H$ are identified with those of $G$ that are trivial on $H$. 
	\end{corollary}
	\begin{proof}
		This follows from \Cref{prop:Artin-L_properties} and the decomposition of the regular representation $\rho_{G/H}$ of the quotient group $G/H$.
		\begin{equation*}
			\zeta(Y/_H,s)= L(X,\mathrm{Ind}_{H}^G 1_H,s)=L(X,\rho_{G/H},s)=L\left(X,\bigoplus_{H\subseteq \ker \rho}\rho^{\oplus \dim \rho},s\right)= \prod_{H\subseteq \ker \rho} L(X,\rho,s)^{\dim \rho}. 
		\end{equation*}
	\end{proof}
	\begin{example}\label{exam:fr-fq}
		Let $Y/X=\spec\F_{q^m}/\spec \Fq$ and $\chi\colon \aut(Y/X)\cong C_m \to \Cx$ be a complex character. Denote the Frobenius element by $\Fr$. Then we have 
		\begin{equation*}
			L(\Fq,\chi,s)=\frac{1}{1-\chi(\Fr) q^{-s}}. 
		\end{equation*}
	\end{example}
	\begin{example}\label{rem:katz} 
		When $X$ is a normal integral $\Fq$-scheme, the Artin $L$-function of a Galois representation of $k(X)$ can be expressed as an exponential sum: (see for example \cite{Katz_4lectures})
		\[
		L(X,\rho,s)=\exp\left(\sum_{m=1}^{\infty}\sum_{x\in X(\Fbb_{q^m})}\tr\left[\rho(\Fr_x)|V^{\Ical(x)}\right]\cdot\frac{q^{-ms}}{m}\right) \in 1 + q^{-s} E\llb q^{-s}\rrb,
		\]
		where $\rho$ is an $E$-linear representation for some subfield $E$ of $\Cbb$. 	It follows that the Artin $L$-function $L(X,\rho,s)$ only depends on the character of the representation $\rho$.  Note that the identity in \Cref{ex:function-field} is the special case of the one above when $\rho=1_G$ is the $1$-dimensional representation. 
		
		We see that the two expressions coincide as follows. For each closed point $x\in |X|$ with residue field $\F_{q^{n}}$, set $x_m$ to be the closed point $\spec \F_{q^{nm}}\to \spec \F_{q^{n}}\xrightarrow{x} X$. Notice the Galois conjugates of $x\colon \spec \F_{q^m}\to X$ over $\Fq$ are identified in $|X|$, but they are counted separately in each $X(\Fbb_{q^m})$. Then we have (see also \cite[pages 157 -- 158 and 165 -- 166]{milneLEC}) 
		\begin{align*}
			&\quad \exp\left(\sum_{m=1}^{\infty}\sum_{x\in X(\Fbb_{q^m})}\tr\left[\rho(\Fr_x)|V^{\Ical(x)}\right]\cdot\frac{q^{-ms}}{m}\right)\\
			&=\exp\left(\sum_{m=1}^{\infty}\sum_{x\in X(\Fbb_{q^m})}\tr\left[\rho(\Fr_x)|V^{\Ical(x)}\right]\cdot\frac{\# \kappa(x)^{-s}}{[\kappa(x):\Fq]}\right)\\
			&=\exp\left(\sum_{x\in |X|}\left([\kappa(x):\Fq]\sum_{m=1}^\infty \tr\left[\rho(\Fr_{x_m})\mid V^{\Ical(x)}\right]\cdot \frac{\# \kappa(x_m)^{-s}}{[\kappa(x_{m}):\Fq]}\right)\right)\\
			&=\prod_{x\in|X|} \exp\left(\left.\tr\left[\sum_{m=1}^\infty\rho(\Fr_{x})^m\frac{\# \kappa(x)^{-ms}}{m} \right)\right|V^{\Ical(x)}\right]\\
			&=\prod_{x\in|X|}\det\left[\left.\exp\left(\sum_{m=1}^\infty \frac{\left(\rho(\Fr_{x})\cdot \# \kappa(x)^{-s}\right)^m}{m}\right)\right|V^{\Ical(x)}\right]\\
			&=\prod_{x\in|X|}\det\left[\exp\left(-\log(1-\rho(\Fr_{x})\cdot \#\kappa(x)^{-s})\right)\mid V^{\Ical(x)}\right]\\
			&=\prod_{x\in|X|}\frac{1}{\det\left[1-\rho(\Fr_x)\cdot \#\kappa(x)^{-s}\mid V^{\Ical(x)}\right]}.
		\end{align*}
	\end{example}
	\begin{example}
		Let $X$ be an integral normal $\Fq$-scheme of dimension $d$, $E$ be a number field, and $\rho\colon \piet(X)\to \GL_N(\Ocal_E)\hookrightarrow \GL_N(\overline{\Q}_\ell)$ be a Galois representation for some prime $\ell\nmid q$. Denote by $\mathcal{F}_\rho$ the lisse sheaf on $X_{\et}$ associated to $\rho$ and $\pi\colon \overline{X}\to X$ is the base change map to $\overline{\F}_q$. Then the Lefschetz fixed point theorem for \'etale cohomology implies 
		\begin{equation}\label{eqn:Weil_L-et_coh}
			L(X,\rho,s)= \prod_{r=0}^{2d}\det\left[1-\Fr \cdot q^{-s}\mid \H^r_\et(\overline{X}; \pi^\ast\Fcal_\rho)\right]^{(-1)^{r+1}},
		\end{equation}
		It follows that $L(X,\rho,s)$ is a rational function of $q^{-s}$ whose coefficients are algebraic numbers in $F$. In particular, $L(X,\rho,s)$ admits a meromorphic continuation to $\Cbb$.
	\end{example}
	\begin{lemma}\label{lem:rational_Weil_L}
		Let $E$ be a number field and $\rho\colon \piet(X)\to \GL_N(\Ocal_E)$ be an $E$-linear Galois representation of a connected, normal, separated $\Fq$-scheme of dimension $d$. Then 
		\begin{enumerate}
			\item For any integer $n$, if $s=n$ is not a pole for $L(X,\rho,s)$, then $L(X,\rho,n)\in E$. 
			\item Let $\sigma\in \mathrm{Aut}_\Q(E)$ be a field automorphism, then $L(X,\sigma\circ \rho,n)=\sigma(L(X,\rho,n))$.
		\end{enumerate}
	\end{lemma}
	\begin{proof}
		(1) follows from the identity \eqref{eqn:Weil_L-et_coh} and (2) is a result of the definition of $L(X,\rho,s)$ in \Cref{rem:katz} as the exponential of a formal power series with coefficients in $F$. 
	\end{proof}
	\begin{proposition}\label{prop:rational_Weil_L_psuedo}
		Let $\pi\colon Y\to X$ be a pseudo-Galois cover of normal integral curves  over $\Fq$ with $G=\gal(k(Y)/k(X))$ and $\rho\colon G\to \GL_N(\Ocal_E)$ be a Galois representation. Then 
		\begin{enumerate}
			\item The Artin $L$-function $L(X,\rho,s)$ admits a meromorphic continuation to $\C$. 
			\item For any integer $n$, if $s=n$ is not a pole for $L(X,\rho,s)$, then $L(X,\rho,n)\in E$. 
			\item Let $\sigma\in \mathrm{Aut}_\Q(E)$ be a field automorphism, then $L(X,\sigma\circ \rho,n)=\sigma(L(X,\rho,n))$.
		\end{enumerate}
	\end{proposition}
	\begin{proof}
		Notice the branch locus $B$ of $\pi$ is a finite disjoint union of closed points in the curve $X$. By \Cref{def:Artin-L}, we have 
		\begin{equation*}
			L(X,\rho,s)=L(X-B,\rho,s)\cdot\prod_{x\in |B|}\frac{1}{\det[\id-\#\kappa(x)^{-s}\rho(\Fr_x)\mid V^{\Ical(x)}]}. 
		\end{equation*} 
		As $Y-\pi^{-1}(B)\to X-B$ is a $G$-Galois cover by assumption,  the three claims follow from \Cref{lem:rational_Weil_L}. 
	\end{proof}
	\begin{example}
		Let $F'/F$ be a $G$-Galois extension of number fields. When $\chi\colon G\to \Cx$ is a $1$-dimensional complex representation (abelian character), the Artin $L$-function of $\chi$ is usually called the \textbf{Dirichlet $L$-function}. The abelian character $\chi$ necessarily factors through an abelian extension of $F$. One particular example is when $F=\Q$. Then the Kronecker-Weber Theorem states any finite abelian Galois extension of $\Q$ is contained in some cyclotomic extension $\Q(\zeta_N)$ of $\Q$.  Let $\chi\colon \znx\cong \gal(\Q(\zeta_N)/\Q)\to \Cx$ be a Dirichlet character, then we have an identity:
		\begin{equation*}
			L(\Z,\chi,s):=\prod_p \frac{1}{1-\chi(p) p^{-s}}=\sum_{n=1}^\infty \frac{\chi(n)}{n^s},\qquad \chi(n)=0\text{ if } (n,N)\ne 1.
		\end{equation*}
		Using global class field theory, there is a similar formula to express the Artin $L$-function of an abelian Galois representation of a number field as a summation  (Hecke $L$-function). See \cite[Theorem VII.10.6]{Neukirch_ANT}. 
	\end{example}
	\begin{proposition}
		Let $F'/F$ be a $G$-Galois extension of number fields.
		\begin{enumerate}
			\item (Artin) The Artin $L$-function $L(\Ocal_F,\chi,s)$ of a non-trivial complex abelian character $\chi\colon G\to \Cx$ extends to an \emph{entire} function of $s$. 
			\item (Brauer) The Artin $L$-function $L(\Ocal_F,\rho,s)$ of any complex Galois representation $\rho$ of $G$ extends to a \emph{meromorphic} function on the complex plane. In particular, this meromorphic continuation is \emph{entire} if $\rho$ is a direct sum of inductions of non-trivial abelian characters on subgroups. 
		\end{enumerate}
	\end{proposition}
	\begin{remark}
		Artin conjectured that for any non-trivial irreducible Galois representation $\rho$ of number fields, the Artin $L$-function $L(X,\rho, s)$ extends to an \emph{entire} function of $s$. He proved this conjecture when $\rho$ is $1$-dimensional using the Artin reciprocity law (1-dimensional Langlands correspondence). See more discussions in \cite{Murty_Artin_L}. 
	\end{remark}
	\begin{remark}
		We note the condition for holomorphic continuation in Brauer's theorem also appears in the assumption of our main \Cref{thm:main}. Indeed, both uses the Brauer Induction Theorem. See more discussions in \Cref{rem:Brauer_ind}.
	\end{remark}
	
	Like \Cref{prop:rational_Weil_L_psuedo}, we also have the following algebraicity result for special values of Artin $L$-functions of number fields.
	\begin{proposition}[{\cite[Theorem 1.2]{Coates-Lichtenbaum}}]\label{prop:alg_special_val_numfld}
		Let $F'/F$ be a $G$-Galois extension of number fields and $\rho\colon G\to \GL_N(\Ocal_E)$ be an $E$-linear Galois representation for some number field $E$. Then $L(\Ocal_F,\rho,1-n)\in E$ if it is not a pole. Moreover if $\rho'=\tau\circ \rho$ for some $\tau\in \gal(E/\Q)$, then $L(\Ocal_F,\rho',1-n)=\tau(L(\Ocal_F,\rho,1-n))$.
	\end{proposition}
	\subsection{M\"obius inversion formulae} \label{subsec:inv_fact} In \Cref{cor:fac_zeta-to-L}, we express zeta functions as products of $L$-functions associated to irreducible representations of Galois groups. Ideally, we want a procedure to extract each $L(X,\chi,s)$ for each character $\chi$ and explain a Quillen--Lichtenbaum phenomenon for it. However, we are unable to separate out characters which are conjugate. In fact, by \cite[Theorem 1.2]{Coates-Lichtenbaum} and \Cref{prop:rational_Weil_L_psuedo}, if the characters are conjugate then their special values are conjugate and therefore the same from the point-of-view of $\K$-theory. Noting that two abelian characters $\chi, \tau: G \rightarrow \Cx$ are conjugate if and only if they have the same kernel, we prove the following key computational tool  which should be thought of as a M\"obius Inversion Formula.  
	\begin{proposition}[M\"obius Inversion Formula]\label{prop:inv_fac}
		Let $G$ be a finite group. Denote the group of complex characters of $G$ by $G^\vee$. For a normal subgroup $H\trianglelefteq G$, we can view $(G/H)^\vee$ as a subgroup of $G^\vee$  by
		\begin{equation*}
			(G/H)^\vee:=\hom(G/H,\Cx)\cong \{\chi\in G^\vee\mid H\subseteq \ker\chi \}.
		\end{equation*}
		Given a family of complex (meromorphic) functions $g(s,\chi)$ indexed by $\chi\in G^\vee$, we define $f(s,G/H)$ for each normal subgroup $H\trianglelefteq G$ by setting 
		\begin{equation}\label{eqn:prod_fac}
			f(s,G/H):=\prod_{\chi\in (G/H)^\vee}g(s,\chi).
		\end{equation}
		Then for any cyclic quotient $G/H\cong C_m$ of $G$ where $m=p_1^{d_1}\cdots p_\lambda^{d_\lambda}$ is the prime factorization of $m$, we have the following M\"obius Inversion Formula:
		\begin{equation*}
			\prod_{\tau\in G^\vee,\ker \tau=H} g(s,\tau)=\prod_{j=0}^\lambda\left(\prod_{1\le k_1<\cdots<k_j\le \lambda} f\left(s,(G/H)/_{C_{p_{k_1}\cdots p_{k_j}}}\right)\right)^{(-1)^j},.
		\end{equation*}
		Here the $j=0$ factor is set to be $f(s,G/H)$.
	\end{proposition}
	\begin{proof}
		On the right hand side of  \eqref{eqn:prod_fac}, separating primitive (injective) characters of $G/H$ from the non-primitive ones, we obtain:
		\begin{equation}\label{eqn:char_prod}
			\prod_{\tau\in G^\vee, \ker \tau=H}g(s,\tau)=\prod_{\tau\colon G/H\hookrightarrow \Cx}g(s,\tau)=f(s,G/H)\cdot\left(\prod_{\tau\in (G/H)^\vee\text{, non-primitive}}g(s,\tau)\right)^{-1}
		\end{equation}
		Notice any non-primitive character $\tau$ of $G/H\cong C_m$ must factor through a further quotient $(G/H)/_{C_{p_i}}\cong C_m/C_{p_i}$ for some prime factor $p_i$ of $m$. By \eqref{eqn:prod_fac}, we have
		\begin{equation}\label{eqn:prod_2factors}
			\prod_{i=1}^\lambda f(s, (G/H)/_{C_{p_i}})=\prod_{i=1}^\lambda\prod_{\tau\in \left((G/H)/_{C_{p_i}}\right)^\vee}g(s,\tau).
		\end{equation}
		Comparing this with \eqref{eqn:char_prod}, we observe the right hand side of \eqref{eqn:prod_2factors} over-counts those characters $\tau\colon G/H\to \Cx$ whose kernels contain $C_{p_{i}p_{j}}$ for some distinct pair $p_i,p_j$ of prime factors of $m$. The M\"obius inversion formula then follows from the Inclusion and Exclusion Principle. 
	\end{proof}
	Next, we  apply \Cref{prop:inv_fac} to the factorization formulas for zeta functions in \Cref{cor:fac_zeta-to-L}. For $L$-functions of (pseudo-)Galois representations of varieties over $\Fq$, we obtain: 
	\begin{corollary}\label{cor:norm_Weil_L}
		Let $Y/X$ be a pseudo-Galois extension of normal integral curves or \'etale covers of integral normal varieties over $\Fq$ for the group $G=C_m$ and $\chi\colon \gal(Y/X)\cong C_m\hookrightarrow \Cx$ be a primitive character. We have
		\begin{equation*}
			\norm_{\Q(\zeta_m)/\Q}L(X,\chi,-n)
			= \prod_{j=0}^\lambda\left(\prod_{1\le k_1<\cdots<k_j\le \lambda} \zeta\left(Y/_{C_{p_{k_1}\cdots p_{k_j}}},-n\right)\right)^{(-1)^j}.
		\end{equation*}
	\end{corollary}
	\begin{proof}
		By \Cref{lem:rational_Weil_L} (for Galois covers of varieties) and  \Cref{prop:rational_Weil_L_psuedo} (for pseudo Galois covers of curves), we have 
		\begin{align*}
			\norm_{\Q(\zeta_m)/\Q}L(X,\chi,-n)=\prod_{\tau\in \gal(\Q(\zeta_m)/\Q)}\tau(L(X,\chi,-n))&=\prod_{\tau\in \gal(\Q(\zeta_m)/\Q)}L(X,\tau\circ \chi,-n)\\&= \prod_{\psi\colon \gal(Y/X)\hookrightarrow \Cx} L(X,\psi,-n).
		\end{align*}
		Setting $G=\gal(Y/X)\cong C_m$ and $g(s,\chi)=L(X,\chi,s)$ in \Cref{prop:inv_fac}, we have $f(s,G/H)=\zeta(Y/_H,s)$ by \Cref{cor:fac_zeta-to-L}. The claim then follows from \Cref{prop:inv_fac}. 
	\end{proof}
	For Dirichlet $L$-functions of number fields,  we have
	\begin{proposition}\label{prop:Dirichlet_L_Dedekind}
		Let $F'/F$ be a $C_m$-Galois extension of number fields. Suppose $m$ has exactly $\lambda$ many distinct prime factors $p_1,\cdots, p_\lambda$. Let $\chi\colon \gal(F'/F)\hookrightarrow \Cx$ be a primitive (injective) character. 
		\begin{enumerate}
			\item Write $\ord_{s=s_0}f(s)$ for the order of vanishing of a meromorphic function $f(s)$ at $s=s_0$. Then 
			\begin{equation*}
				\sum_{\psi\colon \gal(F'/F)\hookrightarrow \Cx} \ord_{s=s_0} L(\Ocal_F,\psi,s)=\sum_{j=0}^\lambda(-1)^j\left(\sum_{1\le k_1<\cdots<k_j\le \lambda} \ord_{s=s_0} \zeta_{(F')^{C_{p_{k_1}\cdots p_{k_j}}}}(s)\right).
			\end{equation*}
			In particular when $s=1-n$, we have:
			\begin{equation*}
				\phi(m)\cdot \ord_{s=1-n} L(\Ocal_F,\chi,s)=\sum_{j=0}^\lambda(-1)^j\left(\sum_{1\le k_1<\cdots<k_j\le \lambda} \ord_{s=1-n} \zeta_{(F')^{C_{p_{k_1}\cdots p_{k_j}}}}(s)\right).
			\end{equation*}
			\item Write $f^*(s_0)$ for the leading coefficient of the Laurent series of a meromorphic function $f(s)$ at $s=s_0$. Then 
			\begin{equation*}
				\prod_{\psi\colon \gal(F'/F)\hookrightarrow \Cx}  L^*(\Ocal_F,\psi,s)=\prod_{j=0}^\lambda\left(\prod_{1\le k_1<\cdots<k_j\le \lambda} \zeta^*_{(F')^{C_{p_{k_1}\cdots p_{k_j}}}}(s)\right)^{(-1)^j},
			\end{equation*}
			In particular, when $F'$ is a totally real number field (then so is $F$) and $s=1-2n$, we have:
			\begin{equation}\label{eq:norm-cal}
				\norm_{\Q(\zeta_m)/\Q}L(\Ocal_F,\chi,1-2n)=\prod_{j=0}^\lambda\left(\prod_{1\le k_1<\cdots<k_j\le \lambda} \zeta_{(F')^{C_{p_{k_1}\cdots p_{k_j}}}}(1-2n)\right)^{(-1)^j}.
			\end{equation}
		\end{enumerate}
	\end{proposition}
	\begin{proof}
		Applying the M\"obius Inversion Formula in \Cref{prop:inv_fac} to factorization of Dedekind zeta function $\zeta_{F'}(s)$ over $F$ in \Cref{cor:fac_zeta-to-L}, we get an identity of meromorphic functions:
		\begin{equation*}
			\prod_{\psi\colon \gal(F'/F)\hookrightarrow \Cx}  L(\Ocal_F,\psi,s)=\prod_{j=0}^\lambda\left(\prod_{1\le k_1<\cdots<k_j\le \lambda} \zeta_{(F')^{C_{p_{k_1}\cdots p_{k_j}}}}(s)\right)^{(-1)^j}.
		\end{equation*}
		The first halves of (1) and (2) follow directly by taking orders of vanishing and leading coefficients of Laurent series at $s=s_0$ on both sides of the identity. \Cref{prop:alg_special_val_numfld} implies that:
		\begin{align*}
			\phi(m)\cdot  \ord_{s=1-n} L(\Ocal_F,\chi,s)&=\sum_{\psi\colon \gal(F'/F)\hookrightarrow \Cx} \ord_{s=1-n} L(\Ocal_F,\psi,s),\\
			\norm_{\Q(\zeta_m)/\Q}L(\Ocal_F,\chi,1-2n)&=\prod_{\psi\colon \gal(F'/F)\hookrightarrow \Cx}  L(\Ocal_F,\psi,1-2n). \qedhere
		\end{align*}
	\end{proof}
	\section{Algebraic $\K$-theory and  the classical Quillen--Lichtenbaum Conjectures} \label{sec:k_ql}
	Let $F/\Q$ be a number field and $\Ocal_\Fbb$ be its ring of integers.  In this section, we review the classical Quillen--Lichtenbaum Conjecture for Dedekind $\zeta$-function attached to $F$. In particular, having pretended to be number theorists in the previous section, we pretend to be homotopy theorists and explain $\K$-theory in a (hopefully) friendly way to number theorists. To orient us, the road from $\zeta$-functions to algebraic $\K$-theory, is summarized in the diagram below (we will elaborate more on it during some proofs in what follows; see \Cref{sec:ql-ff}):
	\begin{equation}\label{eqn:QLC_proof}
		\begin{tikzcd}[column sep=16 ex]
			\substack{\text{Frobenius action on}\\ \H_\et^1(\Ocal_F[1/p,\zeta_{p^\infty}];\Zp(t))}\dar[squiggly]&\substack{\text{$p$-adic $L$-function}\\ \zeta_F(1-t)}\lar[squiggly,"\text{Iwasawa Main Conjecture}"',"\text{Mazur-Wiles}"]\\ \H_c^r\left(\Zpx;\H_\et^s(\Ocal_F[1/p,\zeta_{p^\infty}];\Zp(t))\right)\dar[Rightarrow, "\text{Hochschild-Lyndon-Serre SS}"]& \\
			\H_\et^{r+s}(\Ocal_F[1/p];\Zp(t))\rar[Rightarrow,"\text{Thomason SS}"]&\pi_{2t-r-s}\left(\mathrm{L}_{\K(1)}\K(\Ocal_F[1/p])\right)\\
			\H^{r+s}_{\mot}(\Ocal_F[1/p];\Zp(t))\rar[Rightarrow,"\text{Motivic SS}"]\uar["\text{\'etale-motivic comparison}"]&\pi_{2t-r-s}\K(\Ocal_F[1/p])^\wedge_p\uar["\mathrm{L}_{\K(1)}"']
		\end{tikzcd}
	\end{equation}
	
	We will also discuss the case of finite fields and functions fields of curves. 
	
	\subsection{Reminders on $\K$-theory}\label{sec:k-thy} The zeroth algebraic $\K$-group of a ring, denoted by $\K_0(R)$, is the Grothendieck group of finitely generated projective modules over $R$. While this is a relatively elementary definition, it stores rather subtle information about a ring. For example, say $D$ is a Dedekind domain with fraction field $F$, then the ideal class group of $F$ aka $\mathrm{Pic}(D)$ is a summand of $\K_0(D)$; more precisely we have a canonical short exact sequence
	\[
	0 \rightarrow \mathrm{Pic}(D)  \longrightarrow \K_0(D) \longrightarrow \Z \rightarrow 0;
	\]
	where the first map is defined by observing that any ideal of $D$ is projective. At this point, as already mentioned in the introduction, the reader might know that $\K_1$ has something to do with units. Indeed, if $R$ is a field or, more generally, a semilocal ring then $\K_1(R) \cong R^{\times}$, the group of units in $R$. This could seem unmotivated\footnote{Unless one already knows that one is looking for a higher analog of the class number formula.} but there is a check for the ``correctness'' of the definition of the functor of higher $K$-groups. This comes under the name ``Bass' fundamental theorem''. Indeed, for any ring $R$ (possibly even just an associative, unital ring) one should have a 4-term exact sequence (where the last map is, in fact, a split surjection)
	\[
	0 \rightarrow \K_n(R) \longrightarrow \K_n(R[t]) \oplus \K_n(R[t^{-1}]) \longrightarrow \K_n(R[t, t^{-1}]) \overset{\partial}{\longrightarrow} \K_{n-1}(R) \rightarrow 0. 
	\]
	We refer the reader to \cite[V.8]{weibel-k-book} for further explanation. Hence, once one has defined the functor of $\K_{n}$, one should recover the functor of $\K_{n-1}$. This principle even applies for any theory of negative $K$-groups.
	
	Quillen solved the problem of defining higher $\K$-groups in \cite{Quillen_higher_alg_K}. He built on an earlier definition of his which is somewhat more elementary. Higher $\K$-theory was defined via Quillen's $+$-construction \cite{quillen-plus} which is, informally, the universal way to abelianize the fundamental group of a space without changing its homology; we refer to \cite[IV. 1]{weibel-k-book} for a textbook discussion and \cite{hoyois-plus,raptis,BEHKSY} for modern exposition. The positive homotopy groups of the algebraic $\K$-theory of a ring is then calculated as
	\[
	\K_i(R) := \pi_i(\mathrm{BGL}(R)^+) \qquad i \geq 1,
	\]
	where $\mathrm{BGL}(R) := \colim \mathrm{BGL}_n(R)$. This is sufficient to give meaning to the groups that appear in the Quillen--Lichtenbaum Conjectures. From this point of view, the Quillen--Lichtenbaum Conjectures is ultimately a relationship between the homotopy groups of the ``derived abelianization'' of the classifying space of $\mathrm{GL}_{\infty}(R)$ and zeta functions.
	
	For the purposes of this paper, algebraic $\K$-theory in this guise is not enough. For one, it is not true that $\K(R)$ is functorially equivalent to $\mathrm{BGL}(R)^+ \times \K_0(R)$.\footnote{If $R$ is not regular, then we even have negative homotopy groups in $\K$-theory!} For another, we need certain extra structure of algebraic $\K$-theory that is not immediately visible using the $+$-construction. The viewpoint that we must take is the one taken by Blumberg-Gepner-Tabuada in \cite{blumberg2013universal}. To formulate this perspective, we set $\Cat^{\mathrm{perf}}_{\infty}$ to be the $\infty$-category of small, stable idempotent $\infty$-categories. This $\infty$-category houses the $\infty$-category of perfect complexes on rings in the sense that there is a functor
	\[
	\Perf\colon\CAlg \longrightarrow \Cat^{\mathrm{perf}}_{\infty}.
	\]
	Algebraic $\K$-theory is then a universal functor
	\[
	\K\colon \Cat^{\mathrm{perf}}_{\infty} \longrightarrow \Spt,
	\]
	which preserves final objects and bifiber sequences in $\Cat^{\mathrm{perf}}_{\infty}$ (Verdier sequences); we refer to \cite[Appendix A]{herm-2} for a modern exposition. 
	
	This viewpoint unveils clearly certain structures and properties that algebraic $K$-theory enjoys. One that is relevant for this paper are ``wrong-way'' maps for $\K$-groups. For example, if $f\colon R \rightarrow S$ is a morphism of rings which is perfect  \cite[\href{https://stacks.math.columbia.edu/tag/0685}{Tag 0685}]{stacks-project} and finite  so that the forgetful functor on derived categories preserves perfect complexes \cite[\href{https://stacks.math.columbia.edu/tag/0B6G}{Tag 0B6G}]{stacks-project}: 
	\[
	f_*\colon \Perf(S) \longrightarrow \Perf(R),
	\]
	then we have an induced map on $\K$-groups $\K_*(S) \rightarrow \K_*(R)$. Of course such a structure could be defined using more classical methods, but we will need to keep track of the coherences on this construction as we do via the language of spectral Mackey functors in \Cref{subsec:spectral-mackey-K}. 
	
	The last thing that the reader should keep in mind, though perhaps not strictly needed for the purposes of this paper, is the relationship between $\K$-theory and \'etale/motivic cohomology. The relevance of motivic cohomology, then still a conjecture object, to values of zeta functions was posited by Beilinson \cite{beilinson} and Lichtenbaum \cite{lichtenbaum}. For us, these form the $E_2$ page of the \textbf{motivic spectral sequence}:
	\[
	\H^{i-j}_{\mot}(X; \Z(-j)) \Longrightarrow \K_{-i-j}(X).
	\]
	Here $X$ is a smooth scheme over field, in which case this spectral sequence is the work of many mathematicians, beginning with Bloch and Bloch--Lichtenbaum \cite{bloch-lichtenbaum}, and then later Friedlander-Suslin \cite{friedlander-suslin}, Levine \cite{levine2008homotopy, levine-tech} and Voevodsky \cite{voevodsky293possible}. 
	
	The existence of the above spectral sequence indicates that one can access information stored in these higher $\K$-groups, once we can satisfactorily describe motivic cohomology. For the purposes of $L$-functions, the flagship theorem in this direction is the comparison of motivic cohomology with \'etale/syntomic cohomology as conjectured by Beilinson-Lichtenbaum and proved by to Rost--Voevodsky \cite{voevodsky-zl,voevodsky-z2} away from characteristic and Geisser-Levine \cite{geisser-levine} at the characteristic. 
	
	\subsection{Statements of the classical Quillen--Lichtenbaum Conjectures}\label{sec:ql-ff}
	The most basic case of the Quillen--Lichtenbaum Conjecture comes from Quillen's own computation of the $\K$-groups of finite fields. In \cite{quillen-plus}, he computed the groups as follows:
	\begin{equation}\label{eqn:K_finfld}
		\K_{t}(\Fq)=\begin{cases}
			\Z, & t=0;\\
			\Z/(q^n-1),& t=2n-1>0 ;\\
			0,& \text{else}. 
		\end{cases}
	\end{equation}
	From this, the Quillen--Lichtenbaum Conjecture reads as:
	\begin{equation}\label{eqn:QLC_finfld}
		\zeta(\Fq,-n)=\frac{1}{1-q^n}=-\frac{\# \K_{2n}(\Fq)}{\# \K_{2n-1}(\Fq)}, \qquad n\ge 1.
	\end{equation}
	Notice this identity even ``holds'' at $n=0$ in the sense that $\zeta(\Fq,s)$ has a simple pole at $s=0$, $\K_0(\Fq)=\Z$ is free of rank $1$, and $\K_{-1}(\Fq)=0$.
	
	For smooth affine curves over $\Fq$, the QLC states: 
	\begin{theorem}\label{thm:QLC_curves} Let $X$ be a smooth affine curve over a finite field $\Fq$ with smooth compactification $X\to \wt{X}$. Denote by $r(X)$ the number of closed points in the complement $\wt{X}-X$. Then we have the following equality for $n\ge 2$:
		\begin{equation*}
			\zeta(X,1-n)= \varepsilon\cdot \frac{\# \K_{2n-2}(X)}{\# \K_{2n-1}(X)}, \qquad \varepsilon=(-1)^{r(X)}=(-1)^{\mathrm{rank}(\K_1(X))+1}.
		\end{equation*}
	\end{theorem}
	\begin{proof}[Sketch Proof] We sketch a proof this result in steps:
		\begin{enumerate}
			\item First, we can relate the $\zeta$ functions with \'etale cohomology of $X$ is via the Weil Conjectures, which implies \cite[Theorem 3.1]{Bayer-Neukirch}:
			\begin{equation*}
				|\zeta(X,1-n))|_\ell=\frac{\# \H^2_\et(X;\Zell(n))}{\# \H^1_\et(X;\Zell(n))},
			\end{equation*}
			where $\ell$ is a prime not dividing $q$ and $|-|_\ell$ means $\ell$-adic absolute value (or equivalently valuation). 
			\item Thomason's hyperdescent theorem \cite{thomason_1985} produces a spectral sequence whose $E_2$-terms are \'etale cohomology and abuts to the homotopy groups of $\K(1)$-local algebraic $\K$-theory of $X$:
			\begin{equation*}\label{et}
				\H^{i-j}_\et(X;\Zell(-j))\Longrightarrow \pi_{-i-j}\left(\mathrm{L}_{\K(1)}\K(X)\right).
			\end{equation*}
			In the case of curves, this spectral sequence collapses.  Thomason observes in \cite[(4.6)]{thomason_1985} that:
			\begin{equation*}
				|\zeta(X,1-n)|_\ell=\frac{\# \pi_{2n-2}\mathrm{L}_{\K(1)}\K(X)}{\# \pi_{2n-1}\mathrm{L}_{\K(1)}\K(X)},
			\end{equation*}
			\item Next, we also have motivic spectral sequence \cite{voevodsky293possible,levine2008homotopy}
			\begin{equation*}\label{mot}
				\H^{i-j}_{\mot}(X;\Z(-j))\Longrightarrow \K_{-i-j}(X),
			\end{equation*}
			which fits into a comparison diagram
			\[
			\begin{tikzcd}
				\H^{i-j}_{\mot}(X;\Zell(-j))\rar[Rightarrow,"\text{Motivic SS}"]\dar["\text{\'etale-motivic comparison}"']&[10ex]\K_{-i-j}(X)^\wedge_{\ell}\dar["\mathrm{L}_{\K(1)}"]\\
				\H^{i-j}_{\et}(X;\Zell(-j))\rar[Rightarrow,"\text{Thomason SS}"]&\pi_{-i-j}\left(\mathrm{L}_{\K(1)}\K(X)\right).
			\end{tikzcd}
			\]
			\item There is then an isomorphism on $E_2$-pages thanks to the Rost--Voevodsky theorem \cite{voevodsky-z2,voevodsky-zl} which yields an isomorphism between $\K(1)$-local $\K$-theory and $\K$-theory in the above range.
			\item It remains to determine the sign. Let $Y$ be a smooth variety of dimension $d$ over $\Fq$. From \Cref{def:hs-weil}, it is clear that $\zeta(Y,s)>0$ when $s$ is a real number greater than $d$ (the bound is for the convergence of the infinite product). When $Y/\Fq$ is smooth proper, the sign of $\zeta(Y,d-n)$ is determined by the functional equation of the Weil zeta function \cite[Remark 27.13]{milneLEC}: 
			\begin{equation*}
				\zeta(Y,s)=\pm  q^{-d\chi(\overline{Y})/2}\cdot q^{\chi(\overline{Y})s}\cdot \zeta(Y, d-s),
			\end{equation*}
			where the sign is positive if $d$ is odd and $\chi(\overline{Y})$ is the Euler characteristic of the $\ell$-adic \'etale cohomology of  $\overline{Y}=Y\times_{\Fq}\spec \overline{\F}_q$. It follows that $\zeta(\wt{X},1-n)>0$ for the smooth compactification $\wt{X}$ of $X$. Again by \Cref{def:hs-weil}, we have
			\begin{equation*}
				\zeta(\wt{X},1-n)=\zeta(X,1-n)\cdot \prod_{x\in |\wt{X}- X|}\frac{1}{1-\#\kappa(x)^{n-1}}.
			\end{equation*}
			This implies $(-1)^{r(X)}\cdot \zeta(X,1-n)>0$ for $n>1$, where $r(X):=\#|\wt{X}-X|$. Write $X=\spec A$.  By  Dirichlet's Unit Theorem (see \cite[Theorem 18]{Cassels_globalfields}) and the result of Bass-Milnor-Serre \cite{Bass-Milnor-Serre}, we have $r(X)=\mathrm{rank}(A^\times)+1=\mathrm{rank}(\K_1(X))+1$.\qedhere
		\end{enumerate}
	\end{proof}
	Let $F$ be a number field. Denote by $\zeta^*_F(s)$ the first non-vanishing coefficient in Taylor expansion of the Dedekind zeta function $\zeta_F$ associated to $F$ at $s$,  by $r_1(F)$ the number of real embeddings of $F$, and by $r_2(F)$ the number of pairs of complex conjugate embeddings of $F$. Let $k$ be an integer greater than $1$. 
	\begin{theorem}[Quillen, Borel,  {\cite{Borel_cohomologie,Borel_stable_real_cohomology,Quillen_fin_gen}}]\label{thm:Borel_rank}
		Algebraic $\K$-groups of $\Ocal_F$ are all finitely generated abelian groups. In even degrees, $\K_{2k}(\Ocal_F)$ are finite  when $k>0$. In odd degrees, we have 
		\begin{equation*}
			\dim_\Q \K_{2k-1}(\Ocal_F)\otimes_\Z \Q=\mathrm{rank} (\K_{2k-1}(\Ocal_F))=\ord_{s=1-k}\zeta_F(s)=\begin{cases}
				r_1(F)+r_2(F)-1, & k=1;\\
				r_1(F)+r_2(F), & k>1 \text{ odd};\\
				r_2(F), & k>0 \text{ even}. 
			\end{cases}
		\end{equation*}
	\end{theorem}
	\begin{theorem}[Quillen--Lichtenbaum Conjecture, Rost--Voevodsky]\label{thm:QLC_numfld} 
		The followings hold:
		\begin{enumerate}
			\item  Suppose that $F$ is an \emph{abelian} number field, then the equality
			\begin{equation}\label{eqn:QLC_numfld}
				\zeta^*_{F}(1-k)=\varepsilon \cdot \frac{\# \K_{2k-2}(\mathcal{O}_F)}{\# \K_{2k-1}(\mathcal{O}_F)_{\mathrm{tors}}}\cdot R^B_k(F), \qquad \varepsilon=\begin{cases}
					1, &k\equiv 1\mod 4;\\
					(-1)^{r_1(F)+r_2(F)},&k\equiv 2\mod 4;\\
					(-1)^{r_1(F)}, &k\equiv 3\mod 4; \\
					(-1)^{r_2(F)}, &k\equiv 0\mod 4,
				\end{cases}
			\end{equation}
			holds up to powers of $2$, where $R^B_k(F)$ is the $k$-th Beilinson regulator of $F$.
			\item If $F$ is both \emph{totally real} and \emph{abelian}, then we have the precise equality 
			\begin{equation}\label{eqn:QLC_real_ab_numfld}
				\zeta_{F}(1-2n)=((-1)^{n}\cdot 2)^{r_1(F)}\cdot \frac{\# \K_{4n-2}(\mathcal{O}_F)}{\# \K_{4n-1}(\mathcal{O}_F)},
			\end{equation}
			where $r_1(F)$ is the number of real embeddings of $F$. 
			\item If $F$ is only totally real but not abelian, then \eqref{eqn:QLC_real_ab_numfld} only holds up to possible powers of $2$. 
		\end{enumerate}		
	\end{theorem}
	\begin{remark}\label{rem:r1_K1}
		Like in the affine curve case, the number $r_1(F)$ can also be expressed in terms of algebraic $\K$-groups when $F$ is totally real, using  Dirichlet's Unit Theorem and the Bass-Milnor-Serre Theorem:
		\[r_1(F)=\mathrm{rank}(\Ocal_F^\times)+1=\mathrm{rank}(\K_1(\Ocal_F))+1.\]
	\end{remark}
	\begin{proof}[Proof sketch] We discuss how to find the statements in the literature. Making reference to the diagram \eqref{eqn:QLC_proof}, we note that, thanks to the Rost--Voevodsky theorem, the only bottleneck is the Lichtenbaum Conjecture, which follows from the Iwasawa Main Conjecture (except for the $2$-powers). It was proved by Huber-Kings in \cite[Theorem 1.4.1]{Huber-Kings_2003} for Case (1), by Wiles in \cite[Theorem 1.6]{Wiles-1990} at odd primes for Cases (2) and (3), and by Kolster in \cite[Theorem A.1]{Rognes-Weibel} at the prime $2$ for Case (2), respectively. The precise powers of $2$ in Case (2) are determined by Rognes-Weibel's computation in \cite{Rognes-Weibel}. The signs of special values are determined by the functional equation of $\zeta_F$, see \cite[page 200]{Kolster_k_thy_arithmetic}.
	\end{proof}	
	\section{Equivariant algebraic $\K$-theory and Galois descent}\label{ssec:mack} In this section, begin by discussing the necessary equivariant homotopy theory for the results in this paper. We claim no originality nor completeness in the exposition. We have in mind arithmetic geometers who might not be well-acquainted with the machinery of this theory and we follow the exposition of \cite{Barwick_spectral_Mackey_I}. We will then demonstrate the theory in practice by showing that $K$-theory satisfies descent on the level of homotopy groups in some cases in \Cref{sec:pgal-descent}. One of the points of working with spectral Mackey functors is that these latter statements apply for extension of rings which are not necessarily Galois but only Galois on the fraction fields.

	\subsection{The definition of spectral Mackey functors}\label{subsec:spectral-mackey}

	Let $G$ be a finite group. The basic category in genuine equivariant homotopy theory is the $1$-category $\Orb_G$ of orbits of $G$; objects are transitive $G$-sets which means that they must be of the form $G/H$ for some subgroup $H$, while the morphisms are $G$-equivariant morphisms. This is an enlargement of the basic category in other equivariant mathematics --- $\mathrm{B}G$ is the category with a single object and $G$ worth of morphisms. Thinking of the object in $\mathrm{B}G$ as the orbit $G/e$, we have a fully faithful embedding $\mathrm{B}G \subseteq \Orb^{\op}_G$. 
	
	The $\infty$-category of \textbf{$G$-spaces} is defined to be the presheaf category\footnote{This definition is justified by Elemendorf's theorem \cite{Elmendorf_fixedpt} where he proves that the homotopy theory of $G$-CW complexes is equivalent to this presheaf category.}
	\[
	\Spc_G:=\Fun(\Orb^{\op}_G, \Spc). 
	\]
	Given a $G$-space $\sX$, we will write 
	\[
	\sX(G/H) =: \sX^H. 
	\]
	This is suggestive notation for what follows. We will extract the following important pieces of data out of $G$-space.
	\begin{enumerate}
		\item the \textbf{underlying $G$-space} $\sX|_{\mathrm{B}G} \in \Fun(\mathrm{B}G, \Spc)$ obtained via restriction and the \textbf{underlying space} $\sX^e$. 
		\item the \textbf{fixed points space} $\sX^G$. 
		\item In general, we have a morphism in $\Spc$
		\[
		\sX^G \rightarrow \lim_{\mathrm{B}G} \sX|_{\mathrm{B}G} (=:\sX^{hG}),
		\]
		where the target is the \textbf{homotopy fixed points space}. 
	\end{enumerate}
	The last point highlights a difference between the theory of $G$-spaces and more ``usual'' versions of equivariant mathematics: the fixed points are \emph{data} and we have comparison maps to what one would otherwise guess as the correct notion of fixed points in homotopy theory. Furthermore, for all subgroups $H \leq G$, we do have the \textbf{$H$-fixed points space} $\sX^H$. This pattern persists as we go to spectra.

	This definition suggests that the $\infty$-category of $G$-spectra should be
	\[
	\Fun(\Orb^{\op}_G, \Spt).
	\]
	This is not quite the correct category because objects in it lack certain transfer maps\footnote{These transfer maps are not just novelties for a $G$-cohomology theory; they imply certain relations between such theories such as the one explained in \Cref{prop:Mack_G_inv}.} expected from a $G$-cohomology theories, nonetheless it is elementary and also elucidates certain important constructions that we will encounter. We will denote the above category of $\Spt_{\Orb_G}$ and refer to it as the $\infty$-category of \textbf{$\Orb_G$-\emph{spectra}}. Just like in the spaces situation, we have:
	\begin{enumerate}
		\item the \textbf{underlying $G$-\emph{spectrum}} $\sX|_{\mathrm{B}G} \in \Fun(\mathrm{B}G, \Spt)$ obtained via restriction and the \textbf{underlying \emph{spectrum}} $\sX^e$. 
		\item the \textbf{fixed points \emph{spectrum}} $\sX^G$. 
		\item In general, we have a morphism in $\Spt$
		\[
		\sX^G \rightarrow \lim_{\mathrm{B}G} \sX|_{\mathrm{B}G} (=:\sX^{hG}),
		\]
		where the target is the \textbf{homotopy fixed points \emph{spectrum}}.
	\end{enumerate}
	
	Now, let $\Fin_G$ be the category of finite $G$-sets; this is obtained from the category $\Orb_G$ by adjoining arbitrary finite coproducts. The $\infty$-category of genuine $G$-spectra is built from the indexing following indexing category: we have the $(2,1)$-category of $\Span_G$ of spans of finite $G$-sets. The objects are finite $G$-sets $T$ and morphisms are spans of finite $G$-sets:
	\begin{equation*}
		\begin{tikzcd}
			& S \ar[swap]{ld}{f}\ar{rd}{g} & \\
			T &   & U,
		\end{tikzcd}
	\end{equation*}
	while composition is obtained by pullback. One feature of this category is that it is \textbf{semiadditive} in that it is pointed admits finite products and coproducts which coincide. There is a functor
	\[
	\Fin_G^{\op} \longrightarrow \Span_G;
	\]
	defined informally by
	\[
	(T \rightarrow S) \longmapsto (S \leftarrow T \xrightarrow{=} T).
	\]
	
	A \textbf{spectral Mackey functor} (following \cite{Barwick_spectral_Mackey_I}) is a an additive functor of $\infty$-categories
	\[
	\sX: \Span_G \rightarrow \Spt. 
	\]
	As always we write $\sX(G/H) =: \sX^H$ for $H$ a subgroup of $G$.
	
	Out of this we extract several important pieces of structure that $\sX$ posses:
	\begin{enumerate}
		\item the \textbf{underlying $G$-spectrum} $\sX|_{\mathrm{B}G} \in \Fun(\mathrm{B}G, \Spt)$ obtained via restriction and the \textbf{underlying spectrum} $\sX^e$. 
		\item the \textbf{fixed points spectrum} $\sX^G$. 
		\item In general, we have a morphism in $\Spt$
		\[
		\sX^G \rightarrow \lim_{\mathrm{B}G} \sX|_{\mathrm{B}G} (=:\sX^{hG})
		\]
		where the target is the \textbf{homotopy fixed points spectrum}.
		\item given $H' \leq H \leq G$, then we have the \textbf{restriction} map
		\[
		\mathrm{Res}_{H/H'}\colon \sX^{H} \longrightarrow \sX^{H'};
		\]
		\item and the \textbf{transfer} map
		\[
		\mathrm{tr}_{H'/H}\colon \sX^{H'} \longrightarrow \sX^{H}
		\]
		and $\mathrm{Res}$ and $\mathrm{tr}$ are subject to base change relations, up to coherent homotopies, governed by the span category; see \cite[Summary]{Barwick_spectral_Mackey_I} for details.
	\end{enumerate}
	
	\subsection{Spectral Mackey functors from algebraic $\K$-theory}\label{subsec:spectral-mackey-K} Fix a pseudo $G$-Galois cover $\pi: Y \rightarrow X$; in the situation that we are interested in, $X$ and $Y$ are both one-dimensional, integral normal schemes. The goal of this section is to build a spectral Mackey functor in \Cref{constr:span-k}
	\[
	\underline{\K}_{\pi}:\Span_G \overset{\K}{\longrightarrow} \Spt;
	\]
	which takes the following values: for any $H \leq G$, we have
	\[
	\underline{\K}_{\pi}(G/H) = \K(\Perf(Y/_H)).
	\] For any $H' \leq H \leq G$, we have the restriction map
	\[
	p^{\ast}\colon \K(\Perf(Y/_H)) \longrightarrow \K(\Perf(Y/_{H'})) 
	\]
	induced by the canonical map $p:Y/H' \to Y/H$ and we have the transfer map
	\[
	p_{\ast}\colon \K(\Perf(Y/_{H'})) \longrightarrow \K(\Perf(Y/_H)). 
	\]
	These maps will satisfy a coherent system of compatibilities.
	\begin{remark}
		When $\pi\colon Y\to X$ is a $G$-\emph{Galois} cover, Barwick and Merling constructed $\K(Y)$ as a spectral $G$-Mackey functor (genuine $G$-spectrum) in \cite{Barwick_spectral_Mackey_I,Merling_equivar_alg_k}. The point of this subsection is to extend their results to pseudo Galois coverings in $1$-dimensional cases.
	\end{remark}
	For the rest of this subsection, we fix $X$ a one-dimensional, noetherian integral normal scheme, $G$ a finite group and $\pi\colon Y \rightarrow X$ a pseudo $G$-Galois cover; in particular $Y$ is also one-dimensional, noetherian and integral. We write $k(Y)$ and $k(X)$ as the fraction fields of $Y$ and $X$, respectively. Denote by $\widetilde{\pi}\colon k(X) \rightarrow k(Y)$ the induced map on fraction fields; the latter map is a $G$-Galois extension. For an intermediate extension $E$ between $k(X)$ and $k(Y)$ such that $k(Y)/E$ is a $H$-Galois extension, the normalization $X$ into $E$ is isomorphic to the quotient $Y/_H$ by the following:
	\begin{lemma}\label{lem:explicit}
		Let $A\to A'$ be a pseudo $G$-Galois extension of integral domains. Then for any subgroup $H\le G$, the integral closure of $A$ in $\Frac(A')^H$ is $(A')^H$. 
	\end{lemma}
	\begin{proof}
		On one hand, $(A')^H$ is integral over $A$ as subring of an integral extension $A'$ of $A$. On the other hand, any element $x\in \Frac(A')^H\subseteq \Frac(A')$ integral over $A$ is necessarily in $ \Frac(A')^H\cap A'=(A')^H$, since $A'$ is the integral closure of $A$ in $\Frac(A')$. 
	\end{proof}
	The usual Galois correspondence furnishes an equivalence between the orbit category of $G$ and the category of intermediate Galois extensions between $k(Y)$ and $k(X)$:
	\[
	\Orb_G \simeq \mathrm{Gal}_{k(Y)/k(X)}, \qquad G/H \mapsto (k(Y) \times G/H)^G \cong k(Y)^H. 
	\]
	The category $\mathrm{Gal}_{k(Y)/k(X)}$ is furthermore equivalent to the category of intermediate pseudo Galois extensions.
	
	\begin{definition}\label{def:psgal} Let $\pi\colon Y \rightarrow X$ be a pseudo $G$-Galois extension. Then the category $\mathrm{PsGal}_{Y/X}$ has as objects noetherian integral normal schemes $Y \rightarrow W \rightarrow X$ such that $Y \rightarrow W$ is a pseudo $H$-Galois extension for some $H$ a subgroup of $G$.
	\end{definition}
	
	\begin{lemma}\label{lem:ps-gal} We have an equivalence of categories $\mathrm{PsGal}_{Y/X} \simeq \mathrm{Gal}_{k(Y)/k(X)}$.
	\end{lemma}
	
	\begin{proof} The mutually inverse functors are given by:
		\begin{align*}
			\mathrm{PsGal}_{Y/X} \longrightarrow  \mathrm{Gal}_{k(Y)/k(X)}&&&Y' \rightarrow Y'' \longmapsto k(Y') \rightarrow k(Y'');\\
			\mathrm{Gal}_{k(Y)/k(X)} \longrightarrow \mathrm{PsGal}_{Y/X} &&& F \rightarrow F' \longmapsto \widetilde{X}_{F} \rightarrow  \widetilde{X}_{F'},
		\end{align*}	
		where $\widetilde{X}_F$ is the normalization of $X$ into $F$  computed via \Cref{lem:explicit}. 
	\end{proof}
	On the other hand, taking finite-coproduct completion everywhere gives equivalences 
	\[
	\mathrm{Fin}_G \simeq \FEt_{k(Y)/k(X)} \simeq \mathrm{PsGal}^{\coprod}_{Y/X}
	\]
	Therefore taking spans everywhere gives equivalences
	\[
	\Span_G\simto \Span(\FEt_{k(Y)/k(X)}) \simto \Span(\mathrm{PsGal}^{\coprod}_{Y/X}). 
	\]
	To define transfer for $\K$-theory in our context, let us recall that $f\colon Y \rightarrow X$ is a proper morphism of finite tor-dimension, then the pushforward $f_{\ast}\colon \mathrm{D}_{\mathrm{qc}}(Y) \rightarrow \mathrm{D}_{\mathrm{qc}}(X)$ preserve perfect complexes \cite[\href{https://stacks.math.columbia.edu/tag/0B6G}{Tag 0B6G}]{stacks-project}
	\[
	f_{\ast}\colon \mathrm{Perf}(Y) \longrightarrow \mathrm{Perf}(X);
	\]
	whence we have a pushforward on $\K$-theory. 
	\begin{remark}
		Normalization maps do not necessarily have to have finite tor-dimension hence we are not guaranteed pushforwards on $\K$-theory. In our setting $Y \rightarrow X$ will be a finite morphism of one-dimensional regular affine schemes. In particular $Y, X$, and other intermediate schemes are all regular and all the maps in sight are finite. Therefore, \cite[\href{https://stacks.math.columbia.edu/tag/068B}{Tag  068B}]{stacks-project} guarantees that all morphisms in sight preserves perfect complexes. We use this observation in the next construction. 
	\end{remark}
	\begin{construction}\label{constr:span-k} Fix $\pi:Y \rightarrow X$ a pseudo $G$-Galois cover of one dimensional, regular affine schemes. We construct a functor 
		\[
		\Perf:\Span(\mathrm{PsGal}^{\coprod}_{Y/X}) \longrightarrow \Cat_{\infty}
		\]
		To do so, we recall from \cite[Corollary D.18.1]{Barwick_spectral_Mackey_I} that we have a Mackey functor
		\[
		\Span(\mathrm{DM}, \mathrm{DM}_{\mathrm{FP}}, \mathrm{DM}) \longrightarrow \Cat_{\infty},
		\]
		where $\mathrm{DM}_{\mathrm{FP}}$ refers to the class of morphisms of (spectral) Deligne-Mumford stacks which are strongly proper and morphisms have finite Tor-amplitude. As explained above, because morphisms $\mathrm{PsGal}$ are all proper and has finite tor-dimension, we have an inclusion of subcategories
		\[
		\Span(\mathrm{PsGal}^{\coprod}_{Y/X}) \subset \Span(\mathrm{DM}, \mathrm{DM}_{\mathrm{FP}}, \mathrm{DM}).
		\]
		Therefore, we get a spectral Mackey functor
		\[
		\underline{\K}_{\pi}\colon\Span_G \simto \Span(\mathrm{PsGal}^{\coprod}_{Y/X}) \xrightarrow{\K \circ \Perf} \Spt.
		\]
		In particular, we can take homotopy groups and obtain abelian-group valued Mackey functors
		\[
		\pi_i\underline{\K}_{\pi}\colon\Span_G \xrightarrow{\K \circ \Perf} \mathsf{Ab}, \qquad G/H \longmapsto \K_i(Y/_H).
		\]
	\end{construction}
	\subsection{(Pseudo-)Galois descent on the level of homotopy groups} \label{sec:pgal-descent}
	In this subsection, we study Galois descent on the level of algebraic $\K$-\emph{groups}. The case is simpler when the group order is inverted.
	\begin{proposition}\label{prop:Mack_G_inv}
		Let $\underline{M}$ be a $G$-Mackey functor for a finite group $G$. Then for any subgroup $H\le G$, the map $\mathrm{Res}\colon \M (G/H)[1/|H|]\to \M (G/e)^H[1/|H|]$ is an isomorphism.
	\end{proposition}
	\begin{proof}
		This is a standard result in the theory of (rational) $G$-Mackey functors.  Consider the restriction and the transfer maps:
		\begin{equation*}
			\begin{tikzcd}
				\mathrm{Res}\colon \M (G/H) \rar[shift left] & \M (G/e) \lar[shift left] \rcolon \tr
			\end{tikzcd}
		\end{equation*}
		Their compositions are given by $\tr\circ \mathrm{Res}=|H|$ and $\mathrm{Res}\circ \tr=\sum_{g\in H} g^*$. It is then straightforward to check that $\mathrm{Res}\colon \M (G/H)[1/|H|]\to \M (G/e)^H[1/|H|]$ is an isomorphism with inverse $\tr_*/|H|$.
	\end{proof}
	Applying \Cref{prop:Mack_G_inv} to the algebraic $\K$-group Mackey functors, we obtain:
	\begin{corollary}\label{cor:K_desc_G_inv}
		Let $Y\to X$ be (pseudo-)Galois covering of schemes for a finite group $G$. Then the  map $i^*\colon \K_*(X)[1/|G|]\to \K_*(Y)^G[1/|G|]$ is an isomorphism.
	\end{corollary}
	In some special cases, this isomorphism holds without inverting the group order.  
	\begin{itemize}
		\item finite Galois extensions of finite fields;
		\item finite pseudo-Galois extensions of affine curves over finite fields or rings of integers of number fields and $*\ge 1$ is \emph{odd}.
	\end{itemize}
	
	For finite fields, the fixed point statement follows from Quillen's computation \eqref{eqn:QLC_finfld} in \cite{quillen-plus}.
	\begin{proposition}\label{cor:finite_fixed}
		For any $m\ge 1$, the induced map of field extension on algebraic $\K$-groups $\K_*(\Fq)\to \K_*(\F_{q^m})$ identifies the source as the Galois fixed points of the target. 
	\end{proposition}
	\begin{proof}
		The  $*=0$ case follows from the fact $\K_{0}(\Fq)=\K_{0}(\F_{q^m})=\Z$ with trivial Galois group action.  The $*=2n>0$ case is trivial since $\K_{2n}(\Fq)=\K_{2n}(\F_{q^m})=0$.
		
		When $*=2n-1$, $\K_{2n-1}(\F_{q^m})$ is a  finite abelian group whose orders are coprime to $q$ for all $m$.  The equalizer sequences \cite{quillen-plus}
		\begin{equation*}
			\begin{tikzcd}
				\mathrm{BGL}(\Fq)^+\rar &\BU\rar[shift left,"\psi^q"] \rar[shift right, "\id"'] &[5ex]\BU
			\end{tikzcd}
		\end{equation*}
		for $\Fq$ and $\F_{q^m}$ are connected by 
		\begin{equation*}
			\begin{tikzcd}
				\mathrm{BGL}(\Fq)^+\rar\dar &\BU\rar["\psi^q-\id"]\dar["\id"]\dar["\id"] &[5ex]\BU\dar["\psi^{q^{m-1}}+\cdots +\psi^q+\id"]\\
				\mathrm{BGL}(\F_{q^m})^+\rar &\BU\rar["\psi^{q^m}-\id "]&\BU.
			\end{tikzcd}
		\end{equation*}
		The induced maps on the long exact sequences of homotopy groups break up into maps between short exact sequences:
		\begin{equation*}
			\begin{tikzcd}
				0\rar &\pi_{2n}(\BU)\rar["q^{n}-1"] \dar["1"] &[5ex]\pi_{2n}(\BU) \dar["\frac{q^{nm}-1}{q^n-1}"] \rar & \K_{2n-1}(\Fq)\rar\dar  &0\\
				0\rar &\pi_{2n}(\BU)\rar["q^{nm}-1"]  &\pi_{2n}(\BU)  \rar & \K_{2n-1}(\F_{q^m})\rar&0.
			\end{tikzcd}
		\end{equation*}
		This shows that the natural map $ \K_{2n-1}(\Fq)\to  \K_{2n-1}(\F_{q^m})$ is injective. We next show it is the Galois fixed point. As taking fixed points commutes  completion, we can prove the statement by completing each prime $\ell$ not dividing $q$. Suslin's theorem identifies $\K(\overline{\F}_q)^\wedge_\ell$ with $\KU^\wedge_\ell$ such that Galois actions by the $q$-th power Frobenius $\Fr$ corresponds to the Adams operation $\psi^q$. It follows that $\Fr$ acts on $\K_{2n-1}(\F_{q^m})$ by multiplication by $q^{nm}$, whose fixed points is then precisely $\K_{2n-1}(\Fq) $ from the diagram above.  
	\end{proof}
	For affine curves and rings of integers of number fields (Dedekind domains), the  pseudo-Galois descent on $\K$-groups only holds in odd degrees. For $\K_1$, this follows from the Bass-Milnor-Serre's result in \cite{Bass-Milnor-Serre}.
	\begin{proposition}\label{prop:K1_fixedpt}
		Let $F'/F$ be a finite $G$-Galois extension of global fields, $A=S^{-1}\Ocal_F$ be the ring of $S$-integers in $F$ for a finite set $S$ of primes in $F$, and $A'$ be the integral closure of $A$ in $F'$. Then the induced map $\K_1(A)\to \K_1(A')$ exhibits the source as the $G$-fixed points of the target.
	\end{proposition}
	\begin{proof}
		By  \cite[Corollary 3.4]{Bass-Milnor-Serre}, we have natural isomorphisms $\K_1(A)\cong A^\times$. The claim then follows from the facts that the functor $(-)^\times$ is right adjoint to the group ring functor and $(A')^G=A$ by construction. 
	\end{proof}
	
	For $\K_{2n-1}$ and $n\ge 2$, the fixed point statement is  a result of the following isomorphism in \'etale cohomology groups of Dedekind domains. 
	\begin{proposition}\label{prop:desc_H1et}
		Let $A\to A'$ be an extension of Dedekind domains that is $G$-Galois on their fraction fields for a finite group $G$. Then for any prime $\ell$ away from the characteristic of $A$ and $A'$ and $n\geq 2$, the natural map
		\begin{equation*}
			\H^1_\et(A;\Zell(n))\longrightarrow \H^1_\et(A';\Zell(n))
		\end{equation*}
		exhibits the source as the $G$-fixed point of the target.
	\end{proposition}
	\begin{proof} (See  \cite[Corollary 2.3 and Proposition 2.9]{Kolster_k_thy_arithmetic})
		Let $C$ (resp. $C'$) be the collection of closed points of $\spec A$ (resp. $\spec A'$). Localization and purity in  \'etale cohomology yield a long exact sequence: 
		\begin{equation*}
			0 = \bigoplus_{x \in C} \H_\et^{-1}(\kappa(x); \Zell(n-1)) \longrightarrow	\H^1_\et(A;\Zell(n))\longrightarrow \H^1_\et(\Frac(A);\Zell(n)) \longrightarrow \bigoplus_{x \in C} \H_\et^{0}(\kappa(x); \Zell(n-1)),
		\end{equation*}
		where $\Frac(A)$ is the fraction field of $A$. Since $n\not =1$ the Galois action on $\Zell(n-1)$ on the \'etale site of $\kappa(x)$ for all $x$ is nontrivial, we have that $\H_\et^{0}(\kappa(x); \Zell(n-1)) = 0$, hence we have an isomorphism $\H^1_\et(A;\Zell(n))\cong \H^1_\et(\Frac(A);\Zell(n))$.   As $\Frac(A)\to \Frac(A')$ is a finite $G$-Galois extension of fields, we have a Hochschild-Lyndon-Serre spectral sequence connecting their \'etale cohomology groups:
		\begin{equation*}
			E_2^{s,t}=\H^s(G;\H^t_\et(\Frac(A');\Zell(n)))\Longrightarrow \H^{s+t}_\et(\Frac(A);\Zell(n)). 
		\end{equation*}
		By the same reasoning as in $\kappa(x)$,  $\H^0_\et(\Frac(A');\Zell(n))=0$ whenever $n\ne 0$, hence the edge homomorphism $\H^{1}_\et(\Frac(A);\Zell(n))\to \H^{1}_\et(\Frac(A');\Zell(n))^G$ becomes an isomorphism. The two isomorphisms assemble into the commutative diagram below, proving the isomorphism in the statement. 
		\begin{equation*}
			\begin{tikzcd}[baseline=(B.base)]
				\H^1_\et(A;\Zell(n))\dar["\sim"] \rar &\H^1_\et(A';\Zell(n))^G\dar["\sim"]\\
				\H^1_\et(\Frac(A);\Zell(n))\rar["\sim"]&|[alias=B]| \H^1_\et(\Frac(A');\Zell(n))^G.
			\end{tikzcd}\qedhere
		\end{equation*}
	\end{proof}
	\begin{corollary} \label{cor:fixed-curve}
		Suppose that $Y \rightarrow X$ is a pseudo $G$-Galois cover of smooth affine curves over $\Fq$ or for a finite group $G$. Then the induced map on odd degree algebraic $\K$-groups $\K_{2n-1}(X) \to \K_{2n-1}(Y)$ exhibits the source as the $G$-fixed point of the target when $n\ge 1$. 
	\end{corollary}
	\begin{proof}
		The $n=1$ case has been proved in \Cref{prop:K1_fixedpt}. Harder's theorem \cite{Harder-1977} (see also \cite[Theorem VI.6.1]{weibel-k-book}) states that $\K_*(X)$ is a finite abelian group of order coprime to $q$ when $*\ge 2$. It suffices to prove the isomorphism after completing at each prime $\ell$ not dividing $q$. By Rost--Voevodsky and the Thomason spectral sequence, we have natural isomorphisms for affine curves when $n\ge 2$:
		\begin{equation*}
			\K_{2n-1}(X)^\wedge_{\ell}\simto \K^\et_{2n-1}(X)^\wedge_{\ell}\simto \H^1_\et(X;\Zell(n)).
		\end{equation*}
		The claim then follows from \Cref{prop:desc_H1et}.
	\end{proof}
	There is a similar result in the number field case:
	\begin{corollary}\label{cor:fixed-numfld_odd}
		Let $A'/A$ be a pseudo $G$-Galois extension of integrally closed subdomains of number fields,\footnote{i.e. rings of $S$-integers for some finite set of primes $S$ in the number field.} then the natural map 
		\begin{enumerate}
			\item $\K_{2n-1}(A)\to \K_{2n-1}(A')$ exhibits the source as the $G$-fixed points of the target when $n= 1$ or both $\Frac(A)$ and $\Frac(A')$ has no real embeddings.
			\item $\K_{2n-1}(A)[1/2]\to \K_{2n-1}(A')[1/2]$ exhibits the source as the $G$-fixed points of the target for general pseudo Galois extensions $A'/A$. 
		\end{enumerate}
	\end{corollary}
	\begin{proof}
		Again, the $n=1$ case has been established in \Cref{prop:K1_fixedpt}. By Rost--Voevodsky and the Thomason spectral sequence, we have a natural isomorphism when $n\ge 2$ and $\ell$ is odd or $\Frac(A)$ is totally imaginary (by \cite[Theorem 0.4 and Proposition 6.12]{Rognes-Weibel}):
		\begin{equation*}
			\K_{2n-1}(A)^\wedge_\ell\cong \H^1_\et(A;\Zell(n)).
		\end{equation*}
		The claim then follows from  \Cref{prop:desc_H1et}. 
	\end{proof}
	The $2$-primary fixed point statement is more complicated to prove.
	\begin{proposition}\label{prop:fixed-numfld_2}
		Suppose the fraction field $F'$ of $A'$ in \Cref{cor:fixed-numfld_odd} is a \emph{totally real} number field (then so is $F=\Frac(A)$).  Then the natural map $\K_{2n-1}(A)\to \K_{2n-1}(A')$ exhibits the source as the $G$-fixed point of the target when $n\ge 1$. 
	\end{proposition}
	\begin{proof}
		It remains to study the $2$-primary part of the statement when $n\ge 2$. This was computed by Rognes-Weibel in \cite[Theorem 0.6]{Rognes-Weibel}. Notice that $\K_{2n-1}(A)^\wedge_2\cong \K_{2n-1}(F)^\wedge_2$ by the localization sequence in algebraic $\K$-theory, so we can assume $A=\Ocal_F[1/2]$ and $A'=\Ocal_{F'}[1/2]$ without loss of generality. Denote the set of real embeddings of $F$ by $\mathrm{Emb}(F,\Rbb)$.  This is a $G$-set of size $r_1(F)$. There are four cases. 
		\begin{enumerate}
			\item  $n=4k>0$. The claim follows the same way as in \Cref{cor:fixed-numfld_odd} because of the natural isomorphisms (the case ``$n=8k+7$ of \cite[Theorem 0.6]{Rognes-Weibel}''):
			\begin{equation*}
				\K_{8k-1}(\Ocal_F[1/2])^\wedge_2\simto \H^1_\et(\Ocal_F[1/2];\Z_2(4k)). 
			\end{equation*}
			\item $n=4k+1>1$. We have a trivial extension on the $E_\infty$-page of the motivic spectral sequence (the case ``$n=8k+1$ of \cite[Theorem 0.6]{Rognes-Weibel}''):
			\begin{equation*}
				0\to  \Z_2^{\oplus r_1(F)}\longrightarrow \K_{8k+1}(\Ocal_F)^\wedge_2\longrightarrow \Z/2\to 0.
			\end{equation*}
			It suffices to prove the fixed point statement on both two summands. The $G$-action on $\Z/2$ is necessarily trivial and $(\Z/2)^G=\Z/2$. 
			
			The $\Z_2^{\oplus r_1(F)}$-term is coming from $\left[\K_{8k+1}(\Rbb)^\wedge_2\right]^{\oplus\mathrm{Emb}(F,\Rbb)}$.  Hence it is Galois-equivariantly isomorphic to $\hom_{\mathsf{Set}}\left(\mathrm{Emb}(F,\Rbb),\Rbb\right)$. As $F'/F$ is a $G$-Galois extension, we have an isomorphism of $G$-sets $\mathrm{Emb}(F',\Rbb)\cong G\times \mathrm{Emb}(F,\Rbb)$. This implies that $\left[\Z_2^{\oplus r_1(F')}\right]^G\cong   \Z_2^{\oplus r_1(F)}$. 
			\item $n=4k+2$. This is similar to the previous case, except that the extension problem is non-trivial (the case ``$n=8k+3$ of \cite[Theorem 0.6]{Rognes-Weibel}''\footnote{In the cited reference, there is a direct decomposition of the $\K$-group of interest. However, we do not know how to understand the Galois action on the summand given by $\Z/2w_{4k+2}(F)$. Instead we have to understand the Galois action on the $E_{\infty}$-page of the motivic spectral sequence which is given in \cite[Theorems 6.7]{Rognes-Weibel}. Note that there is a shift in degrees because this reference speaks about the $2^{\infty}$-torsion groups.}). We have the following diagram with the top and bottom rows exact by \cite[Theorems 6.7 and 6.9]{Rognes-Weibel}: 
			\begin{equation*}\adjustbox{width=\linewidth}{
					\begin{tikzcd}[column sep=small]
						0\to(\Z/2)^{\oplus r_1(F)}\rar \dar["\cong"]& \K_{8k+3}(\Ocal_F[1/2])^\wedge_2\rar\dar & \H^1_\et(\Ocal_F[1/2];\Z_2(4k+2))\dar["\cong"]\to  0&\\ 
						0\to  \left[(\Z/2)^{\oplus r_1(F')}\right]^G\rar \dar& \left[\K_{8k+3}(\Ocal_{F'}[1/2])^\wedge_2\right]^G\rar\dar & \left[\H^1_\et(\Ocal_{F'}[1/2];\Z_2(4k+2))\right]^G\rar\dar&\H^1(G; (\Z/2)^{\oplus r_1(F')})=0\\ 
						0\to  (\Z/2)^{\oplus r_1(F')} \rar& \K_{8k+3}(\Ocal_{F'}[1/2])^\wedge_2\rar & \H^1_\et(\Ocal_{F'}[1/2];\Z_2(4k+2))\to 0.&
				\end{tikzcd}}
			\end{equation*}
			The fixed point statement  holds on the \'etale cohomology part by \Cref{prop:desc_H1et}.  For the left vertical map, the term $(\Z/2)^{\oplus r_1(F)}$ is coming from $\left[\K_{8k+3}(\Rbb)^\wedge_2\right]^{\oplus\mathrm{Emb}(F,\Rbb)}$; similarly for $F'$. This uniform description across all fields yields a $G$-equivariant isomorphism $(\Z/2)^{\oplus r_1(F')} \cong \map_\mathsf{Set}(G, (\Z/2)^{\oplus r_1(F)})$, which implies:
			\begin{itemize}
				\item $(\Z/2)^{\oplus r_1(F)}\simto \left[(\Z/2)^{\oplus r_1(F')}\right]^G$;
				\item $\H^1\left(G; (\Z/2)^{\oplus r_1(F')}\right)\cong \H^1\left(e; (\Z/2)^{\oplus r_1(F)}\right)=0$ by Shapiro's Lemma. 
			\end{itemize}
			The isomorphism $\K_{8k+3}(\Ocal_F[1/2])^\wedge_2\simto \left[\K_{8k+3}(\Ocal_F[1/2])^\wedge_2\right]^G$ then follows from the Five Lemma. 		
			\item $n=4k+3$. There are natural isomorphisms $\K_{8k+5}(\Ocal_F)^\wedge_2\cong \Z_2^{\oplus r_1(F)}\cong \hom_{\mathsf{Set}}\left(\mathrm{Emb}(F,\Rbb),\Z_2\right)$ (The case ``$n=8k+5$ of \cite[Theorem 0.6]{Rognes-Weibel}''). The claim follows similar to the $n=4k+1$ case. \qedhere
		\end{enumerate}
	\end{proof}
	
	\begin{remark}
		We note that Galois descent also holds in even dimensional algebraic $\K$-groups in certain cases. When $A'/A$ is a $G$-\emph{Galois} of rings of integers of number fields or affine curves over $\Fq$,  Kolster proved a ``Galois codescent'' in \cite[Proposition 2.12]{Kolster_k_thy_arithmetic} of \'etale cohomology groups
		\begin{equation*}
			\H^2_\et(A';\Zell(n))_G\simto \H^2_\et(A;\Zell(n)),
		\end{equation*}
		where $\ell\ne 2$ in the number field case and $\ell\nmid q$ in the function field case. As both \'etale cohomology groups and the Galois group are finite, taking Pontryagin duals turns codescent into descent:
		\begin{equation*}
			\H^2_\et(A;\Zell(n))\simto \H^2_\et(A';\Zell(n))^G.
		\end{equation*}
		By Rost--Voevodsky and the motivic spectral sequence, we have natural isomorphisms when $n\ge 2$
		\begin{equation*}
			\H^2_\et(A;\Zell(n))\simto \K_{2n-2}(A)^\wedge_{\ell}.
		\end{equation*}
		This yields the Galois descent on even-degree $\K$-groups when $n\ge 2$
		\begin{equation*}
			\K_{2n-2}(A)\simto \K_{2n-2}(A')^G,\qquad \text{ (invert 2 in the number field case).}
		\end{equation*}
		Note that this argument fails when $A'/A$ is only a \emph{pseudo}-Galois extension or when $\ell=2$ in the number field case, because the first step in proving \cite[Proposition 2.12]{Kolster_k_thy_arithmetic} uses the Tate spectral sequence for Galois extensions. 
	\end{remark}
	\section{Equivariant Moore spectra and its cellular filtration}\label{sec:eMoore}
	This section is the technical heart of the paper and uses crucially the construction of a cellular presentation of equivariant Moore spectra due to Rezk and the second author as in \cite[\S 3.3]{nz_Dirichlet_J}. The first key result here is \Cref{cor:e1-norm} where the norms of $L$-values are incarnated as the ratio between Euler numbers of the $E_1$ page of a spectral sequence  built from a cellular filtration on equivariant Moore spectra. Turning the page once, we compute the Bredon cohomology of these Moore spectra with coefficients in equivariant algebraic $\K$-groups in \Cref{sec:moore}.
	\subsection{Equivariant Moore spectra of abelian characters}
	Let's recall the construction of the equivariant Moore spectra following the second author's thesis \cite{nz_Dirichlet_J}. If $A$ is an abelian group, then the Moore spectrum associated to $A$, denoted usually by $\M A$, is characterized uniquely as a \emph{connective} spectrum whose homology group is given as follows:
	\[
	\pi_{*}(\M A \otimes \H\mathbb{Z}) = \H_{*}(\M A;\Z) = \begin{cases}
		A, & * = 0;\\
		0, & \text{else}.
	\end{cases}
	\]
	In other words we have an equivalence
	\[
	\M A \otimes \H \Z \simeq \H A.
	\]
	Note, however, unlike the formation of Eilenberg-MacLane spectra, the formation of Moore spectra $\M A$ is \emph{not} functorial in $A$; this complication presents itself in the second author's thesis. 
	\begin{definition}\label{def:eMoore}
		Let $A$ be an abelian group with an action by a group $G$ via the homomorphism $\rho\colon G\to \aut(A)$. An \textbf{equivariant Moore spectrum} of $\rho$, if it exists, is a $G$-action on the Moore spectrum $\M A$ such that the induced $G$-action on $\H_0(\M A;\Z)=A$ is isomorphic to $\rho$. In other words, $\M (\rho)$ is a lifting
		\begin{equation*}
			\begin{tikzcd}[column sep=large]
				&& \Spt\dar["\H_*(-;\Z)"]\\ \mathrm{B}G\rar["\rho"']\ar[urr,dashed,"\exists \M (\rho)?"]&\mathsf{Ab}=\mathsf{Ab}^{\{0\}}\rar&\mathsf{Ab}^\Z 
			\end{tikzcd}
		\end{equation*}
		
		Any complex representation $\rho\colon G\to \GL_N(\Cbb)$ of a finite group $G$ necessarily factors through $\GL_N(\Ocal_E)$ for some number field $E$; we abusively call this representation $\rho\colon G\to \GL_N(\Ocal_E)$. An integral equivariant Moore spectrum $\M (\rho)$ of $\rho$, if it exists,  is a $G$-action on the Moore spectrum of $\Ocal_E^{\oplus N}$ such that the induced action on its zeroth homology group is given by $\rho\colon G\to \GL_N(\Ocal_E)$.   
	\end{definition}
	\begin{remark}
		Steenrod asked the question whether $\M (\rho)$ always exists and Carlsson constructed a counterexample for the group $C_2\times C_2$ in \cite{carlsson-cex}.  
	\end{remark}
	\begin{remark}\label{rem:eMoore_uniqueness}
		Even if $\M (\rho)$ exists, the lifting is not unique unless multiplication by the group order of $G$ is an automorphism of $A$. One example is the trivial and sign representations of $C_2$ on $\Z$. Let  $\sigma$ be the sign representation of $C_2$. Then depending on the parity of $n$, any representation sphere of the form $\mathrm{S}^{n(1-\sigma)}$ is an equivariant Moore spectrum for $\Z$ with the trivial or the sign $C_2$-action.  Moreover, Beaudry--Goerss--Hopkins--Stojanoska showed in \cite[Theorem 6.22]{BGHS_dualizing} that the set of homotopy classes of $C_2$-actions on $\mathrm{S}^0$ can be identified as:
		\begin{equation*}
			[\mathrm{B}C_2,\mathrm{B}\GL_1(\mathrm{S}^0)]\overset{\sim}{\longleftarrow}[\mathrm{B}C_2,\mathrm{BO}]=\wt{\KO}^0(\mathbb{RP}^\infty)= \Z_2.
		\end{equation*}
		The Atiyah-Segal completion theorem implies that $\mathrm{S}^{n(1-\sigma)}$ corresponds to the element $\pm n\in \Z\subseteq \Z_2$ on the right hand side. Hence, the sets of homotopy classes of $C_2$-actions on $\mathrm{S}^0$ that induce trivial and sign representations on $\H_0(\mathrm{S}^0;\Z)$ can be identified with $2\Z_2$ and $1+2\Z_2\subseteq \Z_2$, respectively.  
	\end{remark}
	First, we give a construction of integral equivariant Moore spectra for abelian characters $\chi\colon G\to \Cx$ for a finite group $G$. This cellular filtration is essential in our identification of norms special values of $L$-functions with sizes of equivariant algebraic $\K$-groups.
	\begin{construction}[Equivariant Moore spectra for abelian characters, see also {\cite[Construction 3.3.2]{nz_Dirichlet_J}}]\label{con:eMoore}
		Let $\chi\colon G\to \Cx$ be a complex-valued abelian character of a finite group $G$. Then there is a unique positive integer $m$ such that $\chi$ factors as
		\begin{equation}\label{eqn:O_chi}
			\begin{tikzcd}
				\chi\colon G\rar["\phi_\chi",->>]\ar[rr,bend right=30,"\Ocal_\chi"]&C_m\rar["\psi_m"]& \Z[\zeta_m]^\times\rar[hook]&\Cx,
			\end{tikzcd}
		\end{equation}
		where $\psi_m$ sends a generator of $C_m$ to a primitive $m$-th root of unity $\zeta_m$. As a result, it suffices to construct equivariant Moore spectra for the cyclotomic characters $\psi_m\colon C_m\to \Z[\zeta_m]^\times$. In turn, this is constructed as follows: 
		
		\begin{enumerate}
			\item First, following a suggestion of Rezk, we construct the case of $m = p^{\nu}$ by setting $\M (\psi_{p^{\nu}})$ to be the cofiber
			\[
			C_{p^{\nu}+} \rightarrow C_{p^{\nu-1}+} \rightarrow \M (\psi_{p^{\nu}})[1];
			\]
			where the first map is taken in the $\infty$-category of $\Spt^{C_{p^{\nu}}}$. 
			\item Second, take the (external) tensor product\footnote{For groups $H, H'$ there is an external tensor product functor
				\[
				\Spt^{H} \times \Spt^{H'} \xrightarrow{\boxtimes} \Spt^{H \times H'}.
				\]}
			\[
			\M (\psi_m) := \boxtimes_{p|m} \M\left(\psi_{p^{\nu_p(m)}}\right) \in \Spt^{C_m}.
			\]
		\end{enumerate}
	\end{construction}
	One model for integral equivariant Moore spectrum associated to $\chi$  would be $\phi_\chi^*\M (\psi_m)$. For technical reasons explained in the proof later, we will instead set 
	\[
	\M (\chi):= \Dbb((\phi_{\chi^{-1}})^* \M (\psi_m)),
	\] where $\Dbb$ is the $G$-equivariant Spanier-Whitehead dual. A priori, $(\phi_{\chi^{-1}})^* \M (\psi_m)$ is an equivariant Moore spectrum attached to the character $\Ocal_{\chi^{-1}}$. Taking duals inverts the character back to $\Ocal_{\chi}$. From this definition, we can see $\ker\chi$ acts trivially on $\M (\chi)$. This particular choice of equivariant Moore spectra allows us to avoid discussing transfer maps in algebraic $\K$-groups in the computations (see Remarks  \ref{rem:eMoore_homology} and \ref{rem:Bredon_res}). 
	\begin{remark}
		Technically, we can also set $\M(\chi)$ to be $\Dbb((\phi_\chi)^* \M (\psi_m))$ without taking the inverse. As $\chi^{-1}=\overline{\chi}$ is complex conjugate to $\chi$, we have $L(X,\chi^{-1},1-n)=\overline{L(X,\chi,1-n)}$ by \Cref{prop:rational_Weil_L_psuedo} and \Cref{prop:alg_special_val_numfld} --- their norms over $\Q$ are the same. So we do not need to distinguish the integral Moore spectra of $\chi$ and $\chi^{-1}$ for the purpose of proving \Cref{thm:main}. 
	\end{remark}

	For the purposes of this paper, we will evaluate spectra of the form $\map\left(\M (\psi_m), \E\right)$, where $\map$ denote $G$-equivariant maps in the genuine stable category. Towards this end we review the skeletal filtration on $\M (\psi_m)$ or, rather, on $\Sigma^\lambda \M (\psi_m)=\M (\psi_m)[\lambda]$ which first appeared in the second author's thesis \cite[Section 3.3]{nz_Dirichlet_J}. 	
	\begin{construction}[Equivariant cellular filtration on equivariant Moore spectra] \label{constr:cell_ss}
		Let $\lambda$ be the number of distinct prime factors of $m$. We have a finite, increasing filtration
		\[
		\mathrm{sk}_0 \M (\psi_m)[\lambda] \rightarrow \mathrm{sk}_1\M (\psi_m)[\lambda] \cdots \rightarrow \mathrm{sk}_\lambda \M (\psi_m)[\lambda],
		\]
		whose graded pieces 
		\[
		\mathrm{gr}_j \M (\psi_m)[\lambda] :=  \mathrm{cofib}\left( \mathrm{sk}_{j-1} \M (\psi_m)[\lambda] \rightarrow \mathrm{sk}_{j}\M (\psi_m)\right)[\lambda]. 
		\]
		are given by 
		\begin{enumerate}
			\item $\mathrm{gr}_0 \M (\psi_m)[\lambda]\simeq (C_m/C_{p_1\cdots p_\lambda})_+$.
			\item For $0 < j < \lambda$, $\mathrm{gr}_j \M (\psi_m)[\lambda]$ is equivalent to 
			
			\[\left( \coprod_{1\le k_1<\cdots<k_{\lambda-j}\le \lambda}C_m/C_{p_{k_1}\cdots p_{k_{\lambda-j}}} \right)_+[j]. \]
			
			\item There is only one equivariant simplex $C_m$ in top dimension $\lambda$, i.e.
			\[
			\mathrm{gr}_\lambda\M (\psi_m)[\lambda] \simeq (C_{m})_+[\lambda]. 
			\]
		\end{enumerate} 
		The attaching maps of equivariant simplices are given by quotients of orbits; we refer to \cite[Pages 34-35]{nz_Dirichlet_J} for details. 
		
		If $\E$ is a $C_m$-spectrum, then mapping the above filtration into $\E$ (and shifting) we then get another finite filtered object
		\begin{equation*}
			\Map( \mathrm{sk}_\lambda\M (\psi_m), \E) \rightarrow \Map( \mathrm{sk}_{\lambda-1} \M (\psi_m), \E)\rightarrow \cdots \rightarrow \Map( \mathrm{sk}_0\M (\psi_m), \E).
		\end{equation*}
		with graded pieces (in decreasing order)
		\[
		\E^{C_{p_1\cdots p_\lambda}}[\lambda], \cdots,  \coprod_ {1\le k_1<\cdots<k_{\lambda-j}\le \lambda} \E^{C_{p_{k_1}\cdots p_{k_{\lambda-j}}}}[\lambda-j], \cdots, \E.
		\]
		Taking equivariant homotopy groups, we obtain an $E_1$-spectral sequence called the \emph{equivariant Atiyah-Hirzebruch spectral sequence} (eAHSS):
		\begin{equation}\label{eq:cell-ss}
			E_1^{s,t}=\bigoplus_{1\le k_1<\cdots<k_s\le \lambda} \pi_{t}\E^{C_{p_{k_1}\cdots p_{k_{s}}}} \Longrightarrow \pi_{s+t}\map^{C_m}\left( \M (\psi_m), \E\right),\quad d_r^{s,t}\colon E_r^{s,t}\to E_{r}^{s-r,t+r-1}.
		\end{equation}
	\end{construction}
	\begin{remark}\label{rem:moore} 
		The construction of the Moore spectrum $\M (\psi_m)$ is an equivariant enhancement of the Moore spectrum $\M\Z[\zeta_m]$ which takes into account the $C_m$-action. More precisely $\M (\psi_m)^e \simeq M\Z[\zeta_m]$. By construction of the skeletal filtration, we can extract a chain complex:
		\begin{equation*}
			\Z[C_m]\to \bigoplus_i \Z[C_m/C_{p_i}]\to \cdots \to \Z\left[C_m/ C_{p_1\cdots p_\lambda}\right]. 
		\end{equation*}
		whose homology group concentrated in degree $0$. 
	\end{remark}
	\begin{remark}\label{rem:e2} To our knowledge, the spectral sequence in \eqref{eq:cell-ss} was first constructed by Davis-L\"uck in \cite{Davis-Lueck_1998}. Note that by \cite[Page 236]{Davis-Lueck_1998}, the $E_2$-page of the spectral sequence \eqref{eq:cell-ss} computes Bredon cohomology with coefficients in (abelian group valued) Mackey functors:
		\[
		E_2^{s,t} \cong \H^s_{C_m}(\M (\psi_m); \underline{\pi}_t\E).
		\]
		This is analogous to the story non-equivariantly where the cellular filtration on $X$ gives an $E_1$-refinement of the $E_2$-page of the Atiyah-Hirzebruch spectral sequence for computing the value of an generalized cohomology theory on $X$. 
	\end{remark}
	\begin{remark}\label{rem:eMoore_homology}
		If we had set $\M(\chi):=\phi_\chi^*\M(\psi_m)$ without taking the Spanier-Whitehead dual, we will instead obtain a \emph{homological} equivariant Atiyah-Hirzebruch spectral sequence with $E_2$-page:
		\begin{equation*}
			E^2_{s,t} \cong \H_s^{C_m}(\M (\psi_m); \underline{\pi}_t\E)\Longrightarrow \pi_{s+t}^{C_m}(\E\otimes \M(\psi_m)).
		\end{equation*}
	\end{remark}
	\begin{corollary}\label{cor:e1-norm}
		Let $\chi\colon C_m=\aut_X(Y)\hookrightarrow \Cx$ be a primitive character of the pseudo Galois cover $Y\to X$. Consider the equivariant Atiyah-Hirzebruch spectral sequence above for the $C_m$-spectral Mackey functor $\E=\K(Y)$:
		\begin{equation*}
			E_1^{s,t}\Longrightarrow \pi_{t-s}^{C_m}(\K(Y)\otimes \M (\chi)).
		\end{equation*}
		Then in the cases of finite and function fields, we have 
		\begin{equation*}
			\norm_{\Q[\zeta_m]/\Q} L(X,\chi,d-n)=(-1)^{\dim_\Q[(\K_1(Y)\oplus\Z_{\mathrm{triv}})\otimes\chi\otimes \Q]^G}\cdot \frac{\text{Euler number }(E^{*,2n-2d}_1)}{\text{Euler number }(E^{*,2n-1}_1)}.
		\end{equation*}
		In the case of totally real and abelian number fields where $X=\spec \Ocal_{F}$ and $Y=\spec \Ocal_{F'}$, denote the number of real embeddings of $F$ by $r_1(F)$.   We have a precise equality:
		\begin{equation*}
			\norm_{\Q[\zeta_m]/\Q} L(X,\chi,1-2n)=((-1)^n\cdot 2)^{\phi(m)\cdot r_1(F)}\cdot \frac{\text{Euler number }(E^{*,4n-2}_1)}{\text{Euler number }(E^{*,4n-1}_1)}.
		\end{equation*}
		In the case of totally real number fields which are not abelian, then the equality holds up to possible powers of $2$.
	\end{corollary}
	\begin{proof}
		We showcase the proof of the totally real and abelian field case. The other cases are simpler. 
		
		Set $\E$ to be a spectral Mackey functor of the algebraic $\K$-theory of the pseudo-Galois covering $\spec \Ocal_{F'}\to \spec \Ocal_F$, constructed in \Cref{constr:span-k}. Then we obtain the $E_1$-page of an equivariant Atiyah-Hirzebruch spectral sequence as in \eqref{eq:cell-ss}:
		\begin{equation}\label{eqn:eAHSS_E1}
			E_1^{s,t}=\bigoplus_{1\le k_1<\cdots<k_s\le \lambda} \K_{t}\left(\Ocal_{(F')^{C_{p_{k_1}\cdots p_{k_s}}}}\right)\Longrightarrow \pi_{s+t}^{C_m}\map( \M (\psi_m), \K(\Ocal_{F'})),~ d_r^{s,t}\colon E_r^{s,t}\to E_{r}^{s-r,t+r-1}.
		\end{equation}
		Notice $F'/F$ is $C_m$-Galois extension of totally real number fields, we have $r_1((F')^H)=r_1(F')/\# H$ for any subgroup $H\le C_m$. Combining \eqref{eq:norm-cal}, \eqref{eqn:QLC_real_ab_numfld}, and \eqref{eqn:eAHSS_E1}, we have:\\
		\adjustbox{width=\textwidth}{\begin{minipage}{\linewidth}
				\begin{align*}
					&&&\norm_{\Q[\image \chi]/\Q}L(\Ocal_F,\chi,1-2n)\\
					&=&&\prod_{j=0}^\lambda\left(\prod_{1\le k_1<\cdots<k_j\le \lambda} \zeta_{(F')^{C_{p_{k_1}\cdots p_{k_j}}}}(1-2n)\right)^{(-1)^j}\\
					&=&& \prod_{j=0}^\lambda\left(\prod_{1\le k_1<\cdots<k_j\le \lambda} \left(2\cdot (-1)^{n}\right)^{r_1\left((F')^{C_{p_{k_1}\cdots p_{k_j}}}\right)}\cdot \frac{\# \K_{4n-2}\left(\mathcal{O}_{(F')^{C_{p_{k_1}\cdots p_{k_j}}}}\right)}{\# \K_{4n-1}\left(\mathcal{O}_{(F')^{C_{p_{k_1}\cdots p_{k_j}}}}\right)}\right)^{(-1)^j}\\
					&=&&\prod_{j=0}^\lambda\left(\prod_{1\le k_1<\cdots<k_j\le \lambda} \left((-1)^n\cdot 2\right)^{r_1(F')/p_{k_1}\cdots p_{k_j}}\right)^{(-1)^j}\cdot  \frac{\prod_{j=0}^\lambda\left(\prod_{1\le k_1<\cdots<k_j\le \lambda} \# \K_{4n-2}\left(\mathcal{O}_{(F')^{C_{p_{k_1}\cdots p_{k_j}}}}\right)\right)^{(-1)^j}}{\prod_{j=0}^\lambda\left(\prod_{1\le k_1<\cdots<k_j\le \lambda} \# \K_{4n-1}\left(\mathcal{O}_{(F')^{C_{p_{k_1}\cdots p_{k_j}}}}\right)\right)^{(-1)^j}}\\
					&=&& \left((-1)^n\cdot 2\right)^{r_1(F)\phi(m)}\cdot \frac{\text{Euler number }(E^{*,4n-2}_1,d_1)}{\text{Euler number }(E^{*,4n-1}_1,d_1)}.
				\end{align*}
			\end{minipage}
		}
		The proofs for the other cases are almost the same, except for signs in the finite and function field cases. Those  will be postponed to \Cref{prop:eQLC_sign}.
	\end{proof}
	\subsection{Bredon cohomology of equivariant Moore spectra of abelian characters}\label{sec:moore}    From \Cref{cor:e1-norm}, the mere existence of the cellular $E_1$-spectral sequence gives an interpretation of the norm of special values of these $L$-functions in terms of Euler numbers of the said spectral sequence.  In this subsection, we compute these Bredon cohomology groups. The key point is the already established structure results of these Mackey functors --- namely that they are weak fixed point Mackey functors as proved in \Cref{sec:pgal-descent}. 
	
	The equivariant Moore spectra admits a cell structure whose cells are concentrated in non-positive (homological) degrees. Nonetheless, we show that the cohomology with coefficients in an arbitrary Mackey functor is concentrated only degrees $0$ and $-1$. We will employ the following rather non-standard definition: a \textbf{weak fixed point Mackey functor} is an $\mathsf{Ab}$-valued Mackey functor $\underline{A}$ such that the comparison map $\underline{A}(G/H) \rightarrow \underline{A}(G/e)^{H}$ is an isomorphism. 
	\begin{remark}
		Let $A$ be an abelian with an action by a finite group $G$. In \cite[\S 6]{TW-simple_Mackey}, the fixed point Mackey functor of this action is defined by setting $\underline{A}(G/H)=A^H$, the restriction maps being inclusions of fixed points, and transfer maps being relative trace maps. \Cref{prop:Mack_G_inv} says any $G$-Mackey functor is a \emph{weak} fixed point Mackey functor when the group order is levelwise invertible. However, they are not necessarily  fixed point Mackey functors in the sense above, since their transfers maps are not specified. See \cite[Corollary 6.5]{TW-str_Mackey} for a classification of  $G$-Mackey functors valued in $k$-vector spaces where $|G|$ is invertible in the field $k$.  
	\end{remark}

	We will need the following computation of Bredon cohomology of the equivariant Moore spectra.
	\begin{proposition}\label{prop:Bredon_psi_m_fixedpt} Let $\underline{A}$ be a weak fixed point $C_m$-Mackey functor and write $A:=\underline{A}(G/e)$. Then
		\[
		\H_{C_m}^{*}(\M (\psi_m); \underline{A})=\begin{cases}
			A/\left[\cup_{i=1}^\lambda A^{C_{p_i}}\right], & *=0;\\
			0, & \text{else},
		\end{cases} 
		\]
		where $[\cup_{i=1}^\lambda A^{C_{p_i}}]$ is the additive closure of the subset $\cup_{i=1}^\lambda A^{C_{p_i}}$ in $A$. 
	\end{proposition}
	\begin{proof}
		From the description of the $C_m$-CW structure of $\M (\psi_m)$ in \Cref{constr:cell_ss}, we know that the equivariant cellular chain complex to compute this Bredon cohomology is 
		\begin{align*}
			~&\underline{A}(C_m/C_{p_1\cdots p_\lambda})\to \cdots \to \bigoplus_{1\le k_1<\cdots<k_s\le \lambda} \underline{A}(C_m/C_{p_{k_1}\cdots p_{k_s}}) \to\cdots \to\underline{A}(C_m/e)\\
			=& A^{C_{p_1\cdots p_\lambda}} \to\bigoplus_{k=1}^\lambda A^{C_{p_1\cdots \widehat{p_k}\cdots p_\lambda} } \to \cdots \to\bigoplus_{1\le k_1<\cdots<k_s\le \lambda} A^{C_{p_{k_1}\cdots p_{k_s}}} \to\cdots\to  \bigoplus_{k=1}^\lambda A^{C_{p_k}}\to A,
		\end{align*}
		where the degrees of the first and the last terms are $-l$ and  $0$, respectively. The differentials in this complex are alternating sums of inclusions of fixed points. We want to prove cohomology of this complex is concentrated in the top degree $0$. As $C_{p_1\cdots p_\lambda}$ is canonically isomorphic to the product $C_{p_1}\times \cdots \times C_{p_\lambda}$, we have an identification of subsets (groups) of $A$: 
		\[ A^{C_{p_{k_1}\cdots p_{k_s}}} = \bigcap_{j=1}^s A^{C_{p_{k_j}}}.\] 
		In this sense, the complex is a \v{C}ech complex for the partial covering $\cup A^{C_{p_j}}$ of $A$. The claim is then  a result of the following lemma.
	\end{proof}
	\begin{lemma}
		Let $A$ be an abelian group and $\{A_i\mid 1\le i\le n\}$ be a finite collection of subgroups of $A$. Consider the following \v{C}ech cochain complex $\check{C}^\bullet\{A_i\subseteq A\mid 1\le i\le n\}$:
		\begin{align*}
			0\to \bigcap_{i=1}^n A_i\longrightarrow\bigoplus_{k=1}^n\bigcap_{i\ne k} A_i\longrightarrow \cdots \longrightarrow \bigoplus_{1\le k_1<\cdots<k_s\le n} \bigcap_{i=1}^s A_{k_i}\longrightarrow \cdots \longrightarrow \bigoplus_{k=1}^n A_k \longrightarrow A\to 0,
		\end{align*} 
		where the differentials are alternating sums of inclusions of subgroups (intersections). Suppose the complex is graded so that the degree of $A$ is zero. Then 
		\[\H^*(\check{C}^\bullet\{A_i\subseteq A\mid 1\le i\le n\})=\begin{cases}
			A/ [\cup_{i=1}^n A_i], &*=0;\\
			0, & \text{else}.
		\end{cases}\]
	\end{lemma}
	\begin{proof}
		We prove this by induction. When $n=0$, the statement is trivial since the cochain complex is $[0\to A\to 0]$ concentrated in degree $0$. 
		
		Suppose the claim has been verified for collections of $n-1$ subgroups of $A$.  Note that we have a short exact sequence of cochain complexes:
		\begin{align*}
			0\to \check{C}^\bullet\{A_i\subseteq A \mid 1\le i\le n-1\}&\longrightarrow \check{C}^\bullet\{A_i\subseteq A\mid 1\le i\le n\} \\&\longrightarrow\check{C}^\bullet\{(A_i\cap A_n)\subseteq A_n\mid 1\le i\le n-1\}[-1]\to 0. 
		\end{align*}
		By the inductive hypothesis, the only non-zero part of the long exact sequence of cohomology groups of this short exact sequence is 
		\begin{align*}
			0\to \H^{-1}\left(\check{C}^\bullet\{A_i\subseteq A\mid 1\le i\le n\}\right)&\longrightarrow A_n/[\cup_{i=1}^{n-1} (A_n\cap A_i)]\\ &\overset{\iota_n}{\longrightarrow} A/[\cup_{i=1}^{n-1} A_i]\longrightarrow  \H^{0}\left(\check{C}^\bullet\{A_i\subseteq A\mid 1\le i\le n\}\right)\to 0. 
		\end{align*}
		The second term above is equal to $A_n/(A_n\cap[\cup_{i=1}^{n-1}  A_i] )$.  Then it is straight forward to check that the map $\iota_n$ is injective with cokernel $A/[\cup_{i=1}^{n} A_i]$. This proves the inductive step. 
	\end{proof}
	\begin{proposition}\label{prop:Bredon_psi_m_general} Let $\underline{A}$ be a $C_m$-Mackey functor. 
		Then
		\[
		\H_{C_m}^{*}(\M (\psi_m); \underline{A}) = 0, \qquad * \not= 0, -1.
		\]
	\end{proposition}
	\begin{proof}
		By \Cref{prop:Mack_G_inv} and \Cref{prop:Bredon_psi_m_fixedpt}, the claim holds for $\H_{C_m}^{*}(\M (\psi_m); \underline{A}[1/m])$. It remains to verify it after completing at each prime $p$ dividing $m$. Write $m=p^v m'$ where $m'$ is coprime to $p$. Suppose $\underline{A}(C_m/H)$ is $p$-complete for all subgroups $H$ of $C_m$. By construction, $\M (\psi_m)$ sits in a $C_m$-equivariant cofiber sequence:
		\begin{equation}\label{eqn:eq_Moore_cofib}
			\M (\psi_m):=\M (\psi_{p^v})\boxtimes \M (\psi_{m'})\longrightarrow \left(C_{p^v}\right)_+\boxtimes \M (\psi_{m'})\longrightarrow \left(C_{p^v}/C_p\right)_+\boxtimes \M (\psi_{m'})
		\end{equation}
		By the induction formulas for Bredon cohomology, we have isomorphisms:
		\begin{align*}
			\H^*_{C_m}\left(\left(C_{p^v}\right)_+\boxtimes \M (\psi_{m'});\underline{A}\right)& \cong \H^*_{C_{m'}}\left( \M (\psi_{m'});\underline{A}(C_{p^v}\times-)\right),\\
			\H^*_{C_m}\left(\left(C_{p^v}/C_p\right)_+\boxtimes \M (\psi_{m'});\underline{A}\right)& \cong \H^*_{C_{m'}}\left( \M (\psi_{m'});\underline{A}(C_{p^v}/C_p\times -)\right).
		\end{align*}
		Notice both $C_{m'}$-Mackey functors $\underline{A}(C_{p^v}\times -)$ and $\underline{A}(C_{p^v}/C_p\times -)$ are fixed point Mackey functor, since they are assumed to be $p$-complete.  \Cref{prop:Bredon_psi_m_fixedpt} implies the Bredon cohomology groups of $\M (\psi_{m'})$ with coefficients in them are concentrated in degree $0$.  The long exact sequence of the Bredon cohomology of the cofiber sequence \eqref{eqn:eq_Moore_cofib} with coefficients in $\underline{A}$ is then:
		\begin{align*}
			\cdots \to 0 \to \H^{-1}_{C_m}(\M (\psi_m);\underline{A})&\longrightarrow \H^0_{C_{m'}}( \M (\psi_{m'});\underline{A}(C_{p^v}/C_p\times -))\\&\longrightarrow \H^0_{C_{m'}}( \M (\psi_{m'});\underline{A}(C_{p^v}\times -))\longrightarrow \H^{0}_{C_m}(\M (\psi_m);\underline{A})\to 0\to \cdots.
		\end{align*}
		Hence the Bredon cohomology groups $\H^{*}_{C_m}\left(\M (\psi_m);\underline{A}\right)$ are concentrated in degrees $0$ and $-1$. 
	\end{proof}
	\begin{remark}\label{rem:Bredon_res}
		In the computations of Bredon cohomology above, only the formulas for restriction maps in the Mackey functor $\underline{A}$ are required. Transfer maps are (implicitly) used  in proving \Cref{prop:Mack_G_inv}, which is needed in the proof of \Cref{prop:Bredon_psi_m_general}. 
	\end{remark}
	\section{Proof of the Main Theorem}
	We now come to the proof of the main \Cref{thm:main}. As a consequence of the computations of Bredon cohomology of equivariant Moore spectra in \Cref{sec:moore}, we see that the spectral sequence collapses at the $E_2$-page and thus gives us the main result relating equivariant $K$-groups to norms of special values of Artin $L$-functions of primitive abelian characters. Bootstrapping from this base case, we first extend our result to non-primitive characters in \Cref{prop:desc} and then to higher dimensional representations in \Cref{subsec:hd_rep}. 
	
	Our method also yields a generalization of Borel's \Cref{thm:Borel_rank} to Artin $L$-functions as in \Cref{thm:twisted_Borel-1d}  and \Cref{prop:twisted_Borel}, which recovers results of Gross and the second author in \cite{Gross_Artin_L,nz_QB_AL}. 
	\subsection{Assembling the proof for abelian characters}
	We are now ready to assemble the proof of \Cref{thm:main} when $\chi$ is an abelian character. There are three cases: finite, function, and totally real number fields. We will first prove the primitive character cases and then use descent to generalize to non-primitive characters. 
	\begin{theorem}\label{thm:eQLC_finfld}
		Let $\chi\colon \gal\left(\Fbb_{q^m}/\Fq\right)\hookrightarrow\Cx$ be a primitive/injective character. Then 
		\begin{equation*}
			\norm_{\Q[\image \chi]/\Q} L(\Fq,\chi,-n)=\pm \frac{\# \pi_{2n}^{C_m}\left(\K(\Fbb_{q^m})\otimes \M (\chi)\right)}{\# \pi_{2n-1}^{C_m}\left( \K(\Fbb_{q^m})\otimes \M (\chi)\right)} \qquad n \geq 1,
		\end{equation*} 
		where the minus sign is taken only $\chi$ is trivial. 
	\end{theorem}
	\begin{proof}
		Consider the $E_1$-page of the equivariant Atiyah-Hirzebruch spectral sequence:
		\begin{equation*}
			E^1_{s,t}=\bigoplus_{1\le k_1<\cdots<k_s\le \lambda} \K_{t}\left(\F_{q^{m/{(p_{k_1}\cdots p_{k_s})}}}\right)\Longrightarrow \pi_{s+t}^{C_m}\map\left( \M (\psi_m), \K(\F_{q^m})\right), d^r_{s,t}\colon E^r_{s,t}\to E^{r}_{s-r,t+r-1}.
		\end{equation*}
		By \Cref{cor:e1-norm}, we have an equality:
		\begin{equation*}
			\norm_{\Q[\image \chi]/\Q}L(\Fq,\chi,-n)=(-1)^{\dim_\Q \chi^{C_m}}\cdot \frac{\text{Euler number }(E^{*,2n}_1,d_1)}{\text{Euler number }(E^{*,2n-1}_1,d_1)}.
		\end{equation*}	
		The $E_2$-page of the spectral sequence is 
		\begin{equation*}
			E_2^{s,t}=\H_{C_m}^s(\M(\psi_m);\underline{\K}_t(\F_{q^m}))\Longrightarrow  \pi_{s+t}^{C_m}\map\left( \M (\psi_m), \K(\F_{q^m})\right).
		\end{equation*}
		\Cref{cor:finite_fixed} says the algebraic $\K$-groups $\underline{\K}_t(-)$ is a weak fixed point Mackey functor of finite abelian groups. By \Cref{prop:Bredon_psi_m_fixedpt}, we have 
		$\H_{C_m}^s(\M(\psi_m);\underline{\K}_t(\F_{q^m}))\ne 0$ unless $s=0$ and $t>0$ is odd.  It follows that the spectral sequence collapses on the $E_2$-page and we have
		\begin{align*}
			\frac{\text{Euler number }(E^{*,2n}_1,d_1)}{\text{Euler number }(E^{*,2n-1}_1,d_1)}=\frac{\# E_2^{0,2n}}{\# E_2^{0,2n-1}}&=\frac{\# \pi_{2n}^{C_m}\map\left( \M (\psi_m), \K(\F_{q^m})\right)}{\# \pi_{2n-1}^{C_m}\map\left( \M (\psi_m), \K(\F_{q^m})\right)}=\frac{\# \pi_{2n}^{C_m}\left(\K(\Fbb_{q^m})\otimes \M (\chi)\right)}{\# \pi_{2n-1}^{C_m}\left( \K(\Fbb_{q^m})\otimes \M (\chi)\right)}.
		\end{align*}
	\end{proof}
	\begin{corollary}\label{cor:size-ff-1d}
		We have
		\begin{equation*}
			\pi_{*}^{C_m}(\K(\Fbb_{q^m})\otimes \M (\chi))=\begin{cases}
				\Z/\mathrm{Norm}_{\Q[\image \chi]/\Q} L(\Fq,\chi,-n)^{-1}\cong \Z[\image \chi]/ L(\Fq,\chi,-n)^{-1}, & *=2n-1>0;\\
				0, &\text{else.}
			\end{cases}
		\end{equation*}
	\end{corollary}
	\begin{proof}
		Notice $L(\Fq,\chi,-n)^{-1}=1-\chi(\Fr)q^n \in \Z[\image \chi]$ is an algebraic integer. As $\K_{2n}(\Fq)=0$ for any finite field $\Fq$ and $n>0$, we have 
		\begin{equation*}
			\pi_{2n}^{C_m}\map\left( \M (\psi_m), \K(\F_{q^m})\right)\cong \H^0_{C_m}(\M(\psi_m);\underline{\K}_{2n}(\F_{q^m}))=0.
		\end{equation*} 
		The claim then follows from \Cref{thm:eQLC_finfld}.
	\end{proof}
	\begin{remark}
		We note that our computation relies on the particular cell structure of the equivariant Moore spectrum $\M(\chi)$. Interested readers are encouraged to carry out the computation using other models of equivariant Moore spectra. In the $C_2$-extension case, this amounts to computing $\mathrm{RO}(C_2)^\wedge_{(1-\sigma)}$-graded $C_2$-equivariant algebraic $\K$-groups of $\F_{q^2}$ as discussed in \Cref{rem:eMoore_uniqueness}.
	\end{remark}
	Next, we prove the twisted QLC for function and totally real number fields. 
	\begin{theorem}\label{thm:eQLC_numfld}
		Let $Y\to X$ be a pseudo $C_m$-Galois cover of integral, normal, and affine schemes of dimension $d=1$ and $\chi\colon C_m\hookrightarrow \Cx$ be a primitive character. 
		\begin{enumerate}
			\item When $Y$ and $X$ are affine curves over a finite field $\Fq$, we have 
			\begin{equation*}
				\norm_{\Q[\image \chi]/\Q}(L(X,\chi,1-n))= (-1)^{\dim_\Q [(\K_1(Y)\oplus \Z_{\mathrm{triv}})\otimes \chi]^{C_m}} \cdot\frac{\# \pi^{C_m}_{2n-2}\left(\K(Y)\otimes \M (\chi)\right)}{\# \pi^{C_m}_{2n-1}\left(\K(Y)\otimes \M (\chi)\right)}.
			\end{equation*}
			\item When $Y=\spec \Ocal_{F'}$, $X=\spec \Ocal_F$, and $F'/F$ be a $C_m$-Galois extension of totally real number fields. Then the following holds up to powers of $2$:
			\begin{equation*}
				\norm_{\Q[\image \chi]/\Q}(L(\Ocal_F,\chi,1-2n))= (-1)^{n\cdot r_1(F)\cdot \phi(\#\image \chi)}\cdot \frac{\# \pi^{C_m}_{4n-2}\left(\K(\Ocal_{F'})\otimes \M (\chi)\right)}{\# \pi^{C_m}_{4n-1}\left(\K(\Ocal_{F'})\otimes \M (\chi)\right)}.
			\end{equation*}
			\item In addition, if $F'$ is both totally real and \emph{abelian} over $\Q$ (then so is $F$), we have the precise formula:
			\begin{equation*}
				\norm_{\Q[\image \chi]/\Q}(L(\Ocal_F,\chi,1-2n))=((-1)^n\cdot 2)^{r_1(F)\cdot \phi(\#\image \chi)}\cdot \frac{\# \pi^{C_m}_{4n-2}\left(\K(\Ocal_{F'})\otimes \M (\chi)\right)}{\# \pi^{C_m}_{4n-1}\left(\K(\Ocal_{F'})\otimes \M (\chi)\right)}.
			\end{equation*}
		\end{enumerate}
	\end{theorem}
	\begin{proof}
		Consider the equivariant-Hirzebruch spectral sequence \eqref{eqn:eAHSS_E1}:
		\begin{equation*}
			E^1_{s,t}=\bigoplus_{1\le k_1<\cdots<k_s\le \lambda} \K_{t}\left(Y/_{C_{p_{k_1}\cdots p_{k_s}}}\right)\Longrightarrow \pi_{s+t}^{C_m}\map\left( \M (\psi_m), \K(Y)\right), d^r_{s,t}\colon E^r_{s,t}\to E^{r}_{s-r,t+r-1}.
		\end{equation*}
		The $E_2$-page of this spectral sequence is:
		\begin{equation}\label{eqn:eAHSS_E2}
			E_2^{s,t}=\H_{C_m}^s(\M(\psi_m);\underline{\K}_t(Y))\Longrightarrow  \pi_{s+t}^{C_m}\map\left( \M (\psi_m), \K(Y)\right).
		\end{equation}
		By \Cref{cor:fixed-curve}, \Cref{cor:fixed-numfld_odd}, and \Cref{prop:fixed-numfld_2}, $\underline{\K}_{2n-1}(Y)$ is a weak fixed point $C_m$-Mackey functor in all three cases. \Cref{prop:Bredon_psi_m_fixedpt} then says 
		\begin{equation*}
			\H_{C_m}^s(\M(\psi_m);\underline{\K}_t(Y))= 0,\qquad\text{ when }s\neq 0.
		\end{equation*}
		For even algebraic $\K$-groups, we have by \Cref{cor:K_desc_G_inv} and \Cref{prop:Bredon_psi_m_general} that 
		\begin{equation*}
			\H_{C_m}^s(\M(\psi_m);\underline{\K}_t(Y))= 0,\qquad \text{ when }s\neq -1,0.
		\end{equation*}
		It follows that the spectral sequence \eqref{eqn:eAHSS_E2} collapses on the $E_2$-page. This yields an isomorphism
		\begin{equation*}
			\H_{C_m}^s(\M(\psi_m);\underline{\K}_t(Y))\cong  \pi_{2n}^{C_m}\map\left( \M (\psi_m), \K(Y)\right),
		\end{equation*}
		and an extension problems on the $E_\infty$-page:
		\begin{equation*}
			0 \rightarrow \H^{-1}_{C_m}(\M (\psi_m); \underline{\K}_{2n-2}(Y)) \longrightarrow \pi^{C_m}_{2n-1}\map\left( \M (\psi_m), \K(Y)\right)\longrightarrow \H^{0}_{C_m}(\M (\psi_m); \underline{\K}_{2n-1}(Y)) \rightarrow 0. 
		\end{equation*}
		Assembling the above, we have equalities ($n$ even in the number field cases):
		\begin{align*} 
			\frac{\# \pi^{C_m}_{2n-2}\map\left( \M (\psi_m), \K(Y)\right)}{\# \pi^{C_m}_{2n-1}\map\left( \M (\psi_m), \K(Y)\right)}= \frac{\# E_2^{0,2n-2}}{\# E_2^{-1,2n-2}\cdot \# E_2^{0,2n-1}}&= \frac{\# E_2^{0,2n-2}/\# E_2^{-1,2n-2}}{\# E_2^{0,2n-1}}\\&=\frac{\text{Euler number }(E^{*,2n-2}_1,d_1)}{\text{Euler number }(E^{*,2n-1}_1,d_1)}.
		\end{align*}
		The claim then follows from \Cref{cor:e1-norm}. 
	\end{proof}
	For general number fields, the argument above does not work. This is mainly because algebraic $\K$-groups $\K_{2n-1}(\Ocal_F)$ are finite only when $F$ is totally real and $n$ is even by Borel's computation in  \Cref{thm:Borel_rank}. Then our argument using Euler numbers of bounded complexes of \emph{finite} abelian groups fails. Instead, we can recover a generalization of  \Cref{thm:Borel_rank} to $L$-functions by Gross in \cite{Gross_Artin_L}, which was reformulated by the second author in \cite{nz_QB_AL} recently. 
	\begin{theorem}[Gross, Zhang, {\cite{Gross_Artin_L,nz_QB_AL}}, $1$-dimensional case]\label{thm:twisted_Borel-1d}
		Let $F'/F$ be a $C_m$-Galois extension of number fields and $\chi\colon C_m\to \Cx$ be a primitive character. Then we have 
		\begin{equation*}
			\dim_{\Q[\image \chi]}\pi^{C_m}_{*}\left(\K(\Ocal_{F'})\otimes \M(\chi)\otimes \H\Q\right)=\begin{cases}
				\ord_{s=1-n}L(\Ocal_F,s,\chi), & *=2n-1>0;\\
				0, & *=2n>0.
			\end{cases}.
		\end{equation*}	
	\end{theorem}
	\begin{proof}
		Recall a more familiar notion of the Euler characteristic of a bounded complex of \emph{finite dimensional} vector spaces: it is defined to be the alternating sum of the dimensions at each level.  Then we have
		\begin{align*}
			&\quad \ord_{s=1-n}L(\Ocal_F,s,\chi)&&\\&=\frac{1}{\phi(m)}\cdot \sum_{j=0}^\lambda(-1)^j\left(\sum_{1\le k_1<\cdots<k_j\le \lambda} \ord_{s=1-n} \zeta_{(F')^{C_{p_{k_1}\cdots p_{k_j}}}}(s)\right)&&\text{\Cref{prop:Dirichlet_L_Dedekind}}\\
			&=\frac{1}{\phi(m)}\cdot \sum_{j=0}^\lambda(-1)^j\left(\sum_{1\le k_1<\cdots<k_j\le \lambda} \dim_\Q \K_{2n-1}\left(\Ocal_{(F')^{C_{p_{k_1}\cdots p_{k_j}}}}\right)\otimes \Q\right)&&\text{Borel's \Cref{thm:Borel_rank}}\\
			&= \frac{1}{\phi(m)} \cdot \text{Euler characteristic of } E^1_{*,2n-1}\otimes \Q&&\text{\eqref{eq:cell-ss}}\\
			&= \frac{1}{\phi(m)}\cdot \dim_\Q \H^0_{C_m}(\M (\psi_m);\underline{\K}_{2n-1}(\Ocal_{F'})\otimes \Q)&&\text{\Cref{prop:Bredon_psi_m_general}}\\
			&= \frac{1}{[\Q[\image \chi]:\Q]}\cdot \dim_\Q \pi^{C_m}_{2n-1}\left(\K(\Ocal_{F'})\otimes \M(\chi)\otimes \H\Q\right)&& \text{rational eAHSS collapses.}\\
			&= \dim_{\Q[\image \chi]}\pi^{C_m}_{2n-1}\left(\K(\Ocal_{F'})\otimes \M(\chi)\otimes \H\Q\right).&&
		\end{align*}
		Note the rational equivariant homotopy groups $\pi^G_{2n-1}\left(\K(\Ocal_{F'})\otimes \M(\chi )\otimes \H\Q\right)$ have a natural $\Q[\image \chi]$-vector space structure, since $\M(\chi)\otimes \H\Q\simeq \H\Q[\image \chi]$ and the $G$-action on homotopy groups are $\Q[\image \chi]$-linear. See \cite[Lemma 3.13]{nz_Dirichlet_J}.  The vanishing of even equivariant algebraic $\K$-groups follow from the Borel's theorem that $\K_{2n}(\Ocal_F)=0$ for any number field and $n\ge 1$.
	\end{proof}
	Similar to the proof above, we determine the signs in the finite and the function field cases in \Cref{thm:main}.
	\begin{proposition}\label{prop:eQLC_sign}
		Let $Y\to X$ be a pseudo $C_m$-Galois cover of integral, normal, and affine schemes of dimension $d\le 1$ over a finite field $\Fq$ and $\chi\colon C_m\hookrightarrow \Cx$ be a primitive character.  Then:
		\begin{equation*}
			\mathrm{sign}(\norm_{\Q[\image \chi]/\Q}L(X,\chi,d-n))=(-1)^{\dim_\Q [(\K_1(Y)\oplus \Z_{\mathrm{triv}})\otimes \chi\otimes \Q]^{C_m}}.
		\end{equation*}
		In particular, the sign is positive in the finite field case unless the character is trivial, since $\K_1(\Fq)$ is a finite group. 
	\end{proposition}
	\begin{proof}
		By \eqref{eqn:QLC_finfld} and \Cref{thm:QLC_curves}, we have for any subgroup $H\le C_m$:
		\begin{equation*}
			\mathrm{sign}(\zeta(Y/_H,d-n))=(-1)^{\dim_\Q[(\K_1(Y/_H)\oplus\Z_{\mathrm{triv}})\otimes \Q]}.
		\end{equation*}
		Then similar to the proof of \Cref{thm:twisted_Borel-1d}, we have:
		\begin{align*}
			\mathrm{sign}(\norm_{\Q[\image \chi]/\Q}L(X,\chi,d-n))&=\prod_{j=0}^\lambda\left(\prod_{1\le k_1<\cdots<k_j\le \lambda}\mathrm{sign}\left(\zeta\left(Y/_{C_{p_{k_1}\cdots p_{k_j}}},d-n\right)\right)\right)^{(-1)^j}\\
			&=(-1)^{ \sum_{j=0}^\lambda(-1)^j\left(\sum_{1\le k_1<\cdots<k_j\le \lambda} \dim_\Q\left[\left( \K_{1}\left(Y/_{C_{p_{k_1}\cdots p_{k_j}}}\right)\oplus\Z_{\mathrm{triv}}\right)\otimes \Q\right]\right)}\\
			&=(-1)^{\dim_\Q[(\K_1(Y)\oplus\Z_{\mathrm{triv}})\otimes \chi\otimes \Q]^G}.\qedhere
		\end{align*}
	\end{proof}
	\begin{remark}
		Recall in \Cref{cor:e1-norm}, we proved the sign and extra powers of $2$ in the totally real and abelian number field $F'/F$ case is $((-1)^n\cdot 2)^{r_1(F)\cdot \phi(m)}$, where $r_1(F)$ is the number of real embeddings of the base field $F$. Following \Cref{rem:r1_K1}, we note the number $r_1(F)\cdot \phi(m)$ can also be expressed in terms of the first rational equivariant algebraic $\K$-groups:
		\begin{equation*}
			r_1(F)\cdot \phi(m)=\dim_\Q[(\K_1(\Ocal_{F'})\oplus\Z_{\mathrm{triv}})\otimes \chi]^{C_m}
		\end{equation*}
		In the function field case, there is no nice  formula for the sign in terms of $r(X)$, the number of closed points in the complement of smooth completion of the base curve, due to  possible ramifications and residue field extensions at those points in the pseudo-Galois cover. This is not an issue in the totally real number field case, because all real places of a totally real number field $F$ split completely in a Galois extension to another totally real number field $F'$. As a result, we have $r_1(F')=r_1(F)\cdot [F':F]$.
	\end{remark}
	Lastly, we drop the assumption that $\chi\colon C_m\to \Cx$ is primitive from the results in this subsection. 
	\begin{proposition}\label{prop:desc}
		\Cref{thm:eQLC_finfld}, \Cref{thm:eQLC_numfld}, and \Cref{thm:twisted_Borel-1d} hold for any abelian characters of Galois groups $\chi\colon G\to \Cx$. 
	\end{proposition}
	\begin{proof}
		On the $L$-function side, we have a descent identity in \Cref{prop:Artin-L_properties}:
		\begin{equation*}
			L(X,\chi,s)=L(X,\chi',s),
		\end{equation*}
		where $\chi'\colon G/\ker \chi\hookrightarrow \Cx$ is a primitive character of the pseudo-Galois cover $Y/_{\ker \chi}\to X$. On the equivariant $\K$-group side, note that $\ker \chi$ acts trivially on $\M(\chi)$ by \Cref{con:eMoore}. We also have descent equivalences:
		\begin{equation*}
			[\K(Y)\otimes \M(\chi)]^G\simeq [\K(Y)^{\ker \chi}\otimes \M(\chi')]^{G/\ker \chi}\simeq [\K(Y/_{\ker \chi})\otimes \M(\chi')]^{G/\ker \chi}.
		\end{equation*}
		The identity and the equivalence above combine to reduces to the claim to the twisted QLC for the primitive character $\chi'$. 
	\end{proof}
	\begin{corollary}
		When $\chi\colon \znx=\gal(\Q[\zeta_N]/\Q)\to \Cx$ is an even Dirichlet character, we have 
		\begin{align*}
			\norm_{\Q[\image \chi]/\Q}(L(\Z,\chi,1-2n))= ((-1)^n\cdot 2)^{\phi(\#\image \chi)}\cdot \frac{\# \pi^{\znx}_{4n-2}\left(\K(\Z[\zeta_N])\otimes \M (\chi)\right)}{\# \pi^{\znx}_{4n-1}\left(\K(\Z[\zeta_N])\otimes \M (\chi)\right)}.
		\end{align*}
	\end{corollary}
	\begin{proof}
		For an even Dirichlet character $\chi\colon \znx\to \Cx$, we have that $\Q[\zeta_N]^{\ker\chi}$ is a totally real abelian extension of $\Q$. \Cref{prop:desc} reduces the equality to Artin $L$-function of the primitive character $\chi'\colon \znx/\ker\chi\hookrightarrow \Cx$ case, which is proved in \Cref{thm:eQLC_numfld}. 
	\end{proof}
	\subsection{The proof for higher dimensional representations}\label{subsec:hd_rep}
	In the last step to prove the Main \Cref{thm:main}, we bootstrap  results from the previous subsection to higher dimensional Galois representations. First of all, we need to construct their integral equivariant Moore spectra in some nice cases.
	\begin{proposition}\label{thm:eMoore_higherdim}
		Let $E$ be a number field and $\rho\colon G\to \GL_N(\Ocal_E)\hookrightarrow \GL_N(\Cbb)$ be an $E$-linear representation of a finite $G$. If $\rho$ is isomorphic to a sum of inductions of abelian characters on subgroups of $G$. Then there is a $G$-CW complex structure on the Moore spectrum $\M (\Ocal_E^{\oplus N})$ such that induced action on the homology groups is isomorphic to $\rho$. 
	\end{proposition}
	\begin{proof}
		The case of abelian characters  has been explained in \Cref{con:eMoore}. Suppose $\rho=\bigoplus_i \mathrm{Ind}_{H_i}^G \chi_i$ ($H_i$ can be repeated) and each equivariant Moore spectrum $\M (\chi_i)$ of $\chi_i\colon H\to \Z[\image \chi_i]^\times$ has been constructed. Then set 
		\begin{equation}\label{eqn:eMoore}
			\M (\rho):= \bigoplus_i \mathrm{Ind}_{H_i}^G \left(\M (\chi_i)^{\oplus [E:\Q[\image \chi_i]]}\right).
		\end{equation}
		Notice when $\zeta_m\in E$, we have $\Ocal_E$ is $C_{m}$-equivariantly isomorphic to $\Z[\zeta_{m}]^{\oplus [E:\Q[\zeta_m]]}$, where the group acts by multiplication by roots of unity. It follows that $\M (\chi_i)^{\oplus [E:\Q[\image \chi_i]]}$ is an equivariant Moore spectrum for the character $\wt{\chi}_i\colon H\xrightarrow{\chi_i} \Z[\image \chi_i]^\times\hookrightarrow \Ocal_E^\times$ after scalar extensions. Then we have $G$-equivariant isomorphisms:
		\begin{equation*}
			\H_0(\M (\rho);\Z)\cong  \bigoplus_i \mathrm{Ind}_{H_i}^G \left(\H_0(\M (\chi_i);\Z)^{\oplus [E:\Q[\image \chi_i]]}\right)\cong \bigoplus_i \mathrm{Ind}_{H_i}^G \wt{\chi}_i=\rho.\qedhere
		\end{equation*}
	\end{proof}
	\begin{remark}\label{rem:Brauer_ind}
		The Brauer Induction Theorem (see \cite[Theorems 10.20]{Serre_rep1977}) states that the complex representation ring of a finite group $G$ is generated by inductions of abelian characters on  its subgroups. It follows that any complex representation $\rho$ of $G$ can be written as a \emph{virtual difference} of two finite sums of inductions of abelian characters on subgroups. In \cite[Exercise 10.5.(b)]{Serre_rep1977}, Serre gave the following example of an irreducible representation of a finite group that \emph{cannot} be decomposed as a direct sum of inductions of abelian characters on subgroups. 
		
		Let $G=A_5$ be the alternating group on $5$ elements. Then the character of the reduced permutation representation $\rho\colon G\to \GL_4(\Cbb)$ is not a linear combination of inductions of abelian characters with positive real coefficients. 	For this $\rho$, we can construct an integral equivariant Moore spectra via the $A_5$-equivariant cofiber sequence:
		\begin{equation*}
			\Dbb \M (\rho)\longrightarrow \Sigma_+^{\infty} \{1,2,3,4,5\}\longrightarrow  \Sigma_+^{\infty}\{*\},
		\end{equation*}
		where $\{1,2,3,4,5\}$ is an $A_5$-set with permutation action and $\{*\}$ has trivial $A_5$-action. Moreover, one can prove the twisted QLC for this Galois representation using the filtration above, just like the abelian character case. 
		
		For a general irreducible representation of a finite group, however, it is not clear whether an integral equivariant Moore spectrum exists or not. 
	\end{remark}
	The construction of $\M (\rho)$ in \Cref{thm:eMoore_higherdim} allows us to prove the  Main \Cref{thm:main} for higher dimensional Galois representations:
	\begin{theorem}\label{prop:eQLC_reduction}
		Let $G$ be a finite group and $Y\to X$ be a pseudo $G$-Galois covering of integral normal schemes of dimension $d\le 1$. Suppose $\rho\colon G\to \GL_N(\Ocal_E)$ is conjugate to a sum of induction of complex abelian characters on subgroups of $G$. Then \Cref{thm:main} holds for the Artin $L$-function $L(X,\rho,s)$. 
	\end{theorem}
	\begin{proof}
		Like in \Cref{cor:e1-norm}, we showcase the totally abelian number field case. Let $X=\spec \Ocal_F$ and $Y=\spec \Ocal_{F'}$.  Write $\rho=\oplus_i \mathrm{Ind}_{H_i}^G \chi_i$. The identities of Artin $L$-functions in \Cref{prop:Artin-L_properties} imply that
		\begin{equation*}
			L(\Ocal_F,\rho,s)=\prod_i L(\Ocal_F,\mathrm{Ind}_{H_i}^G\chi_i,s)=\prod_i L(\Ocal_{(F')^{H_i}},\chi_i,s),
		\end{equation*}
		where each $\chi_i$ is a character of the pseudo $G/\ker \chi_i$-Galois extension $\Ocal_{(F')^{H_i}}\to \Ocal_{F'}$. Then \Cref{thm:eQLC_numfld} and \Cref{prop:desc} for $L(\Ocal_{(F')^{H_i}},\chi_i,s)$ states that:
		\begin{align}
			\norm_{\Q[\image \chi_i]/\Q} L(\Ocal_{(F')^{H_i}},\chi_i, 1-2n)&=((-1)^n\cdot 2)^{\phi(\#\image \chi_i)\cdot r_1\left((F')^{H_i}\right)}\cdot \frac{\#\pi_{4n-2}^{H_i}\left(\K(\Ocal_{F'})\otimes \M (\chi_i)\right)}{\#\pi_{4n-1}^{H_i}\left(\K(\Ocal_{F'})\otimes \M (\chi_i)\right)}\nonumber\\\label{eqn:norm_ind}
			&=((-1)^n\cdot 2)^{[\Q[\image \chi_i]:\Q]\cdot r_1(F)\cdot [G:H_i]}\cdot \frac{\#\pi_{4n-2}^{H_i}\left(\K(\Ocal_{F'})\otimes \M (\chi_i)\right)}{\#\pi_{4n-1}^{H_i}\left(\K(\Ocal_{F'})\otimes \M (\chi_i)\right)}.
		\end{align}
		It follows that 
		\begin{align*}
			&\quad~ \norm_{E/\Q} L(\Ocal_F,\rho,1-2n))&&\\
			&=\prod_i \norm_{E/\Q} L(\Ocal_{(F')^{H_i}},\chi_i,1-2n)&&\text{\ref{prop:Artin-L_properties}}\\
			&=\prod_i \left[\norm_{\Q[\zeta_{m_i}]/\Q} L(\Ocal_{(F')^{H_i}},\chi_i,1-2n)\right]^{[E:\Q[\image \chi_i]]}&& \\
			&= \prod_i \left(((-1)^n\cdot 2)^{[\Q[\image \chi_i]:\Q]\cdot r_1(F)\cdot [G:H_i]}\cdot \frac{\#\pi_{4n-2}^{H_i}\left(\K(\Ocal_{F'})\otimes \M (\chi_i)\right)}{\#\pi_{4n-1}^{H_i}\left(\K(\Ocal_{F'})\otimes \M (\chi_i)\right)}\right)^{[E:\Q[\image \chi_i]]} && \text{\eqref{eqn:norm_ind}}\\
			&=((-1)^n\cdot 2)^{\sum_i r_1(F)\cdot[E:\Q]\cdot \dim_E(\Ind_{H_i}^G \wt{\chi}_i\otimes \Q)}  \prod_i \left(\frac{\#\pi_{4n-2}^{G}\left(\K(\Ocal_{F'})\otimes \mathrm{Ind}_{H_i}^G  \M (\chi_i)\right)}{\#\pi_{4n-1}^{G}\left(\K(\Ocal_{F'})\otimes \mathrm{Ind}_{H_i}^G  \M (\chi_i)\right)}\right)^{[E:\Q[\image \chi_i]]}&& \\
			&=((-1)^n\cdot 2)^{r_1(F)\cdot \dim_\Q[\bigoplus_i(\Ind_{H_i}^G\wt{\chi}_i\otimes \Q)]} \cdot  \frac{\#\pi_{4n-2}^{G}\left(\K(\Ocal_{F'})\otimes \left(\oplus_i\mathrm{Ind}_{H_i}^G  \M (\chi_i)^{\oplus [E:\Q[\image \chi_i]]}\right)\right)}{\#\pi_{4n-1}^{G}\left(\K(\Ocal_{F'})\otimes \left(\oplus_i\mathrm{Ind}_{H_i}^G  \M (\chi_i)^{\oplus [E:\Q[\image \chi_i]]}\right)\right)} && \\
			&= ((-1)^n\cdot 2)^{ r_1(F)\cdot (\dim_\Q \rho\otimes \Q)}\cdot \frac{\#\pi_{4n-2}^{G}\left(\K(\Ocal_{F'})\otimes \M (\rho)\right)}{\#\pi_{4n-1}^{G}\left(\K(\Ocal_{F'})\otimes \M (\rho)\right)}.&& \text{\eqref{eqn:eMoore}}
		\end{align*}
		Parallel to the proof of \Cref{prop:eQLC_sign}, the proof of the signs in the finite and function field cases are similar to the following computation of rational equivariant algebraic $\K$-groups with coefficients in higher dimensional Galois representations with in \Cref{prop:twisted_Borel}.
	\end{proof}
	\begin{remark}\label{rem:size-ff}
		Like in \Cref{cor:size-ff-1d}, we can determine the structure of equivariant algebraic $\K$-groups of finite fields with coefficients in higher dimensional Galois representations, though they are not necessarily cyclic. In this case, any Galois representation $\rho\colon \gal(\F_{q^m}/\Fq)=C_m\to \GL_N(\Ocal_E)$ is isomorphic to a direct sum of abelian characters over $\Q[\zeta_m]\subseteq \Cbb$. So an integral equivariant Moore spectra $\M(\rho)$ can always be constructed as in \Cref{thm:eMoore_higherdim}. 
	\end{remark}
	Rationally, we can similarly generalize \Cref{thm:twisted_Borel-1d} to higher dimensions.
	\begin{theorem}[Gross, Zhang, {\cite{Gross_Artin_L,nz_QB_AL}}, higher dimensional case]\label{prop:twisted_Borel}
		Let $F'/F$ be a Galois extension of number fields and $\rho\colon G\to\GL_N(\Ocal_E)$ be an $E$-linear Galois representation for some number field $E$. Then there exists a unique rational equivariant Moore spectra $M\Q(\rho)$ such that 
		\begin{equation*}
			\ord_{s=1-n}L(\Ocal_F,\rho,s)=\dim_E \pi^G_{2n-1}(\K(\Ocal_{F'})\otimes M\Q(\rho)).
		\end{equation*}
	\end{theorem}
	\begin{proof}
		For a $\Q$-module $A$, its Moore spectrum $\M A$ is equivalent to its  Eilenberg-MacLane spectrum $\H A$. This implies the existence and uniqueness of rational equivariant Moore spectra. See more discussions in \cite{nz_QB_AL}.  We will deduce the identity from the $1$-dimensional case in \Cref{thm:twisted_Borel-1d}. By Brauer Induction Theorem, we can write $\rho=\bigoplus_i \mathrm{Ind}_{H_i}^G \chi_i - \bigoplus_i \mathrm{Ind}_{H'_j}^G\chi'_j$. From this decomposition, we can set 
		\begin{equation}\label{eqn:MQrho}
			M\Q(\rho):=\mathrm{Cofib}\left[\bigoplus_j \mathrm{Ind}_{H'_j}^G M(\chi'_j)^{\oplus [E:\Q[\image \chi'_j]]}\otimes \H\Q\longrightarrow\bigoplus_i \mathrm{Ind}_{H_i}^G M(\chi_i)^{\oplus [E:\Q[\image \chi_i]]}\otimes \H\Q \right],
		\end{equation} 
		where the map between Moore/Eilenberg-MacLane spectra is induced by the inclusion of $E$-vector spaces from the decomposition of $\rho$ in the Brauer Induction Theorem. Then we have:
		\begin{align*}
			&\quad~\ord_{s=1-n}L(\Ocal_F,\rho,s)\\&= \sum_j\ord_{s=1-n} L(\Ocal_{(F')^{H'_j}},\chi'_j,s)-\sum_i\ord_{s=1-n} L(\Ocal_{(F')^{H_i}},\chi_i,s)\\
			&= \sum_j\dim_{\Q[\image \chi'_j]}\pi^{H'_j}_{2n-1}(\K(\Ocal_{F'})\otimes M(\chi'_j)\otimes \H\Q)-\sum_i\dim_{\Q[\image \chi_i]}\pi^{H_i}_{2n-1}(\K(\Ocal_{F'})\otimes M(\chi_i)\otimes \H\Q)\\
			&=  \dim_{E}\pi^{G}_{2n-1}\mathrm{Cofib}\left[\bigoplus_j \mathrm{Ind}_{H'_j}^G M(\chi'_j)^{\oplus [E:\Q[\image \chi'_j]]}\otimes \H\Q\longrightarrow\bigoplus_i \mathrm{Ind}_{H_i}^G M(\chi_i)^{\oplus [E:\Q[\image \chi_i]]}\otimes \H\Q \right]\\
			&= \dim_E \pi^G_{2n-1}(\K(\Ocal_{F'})\otimes M\Q(\rho)).
		\end{align*}
		In the second last step, we used the facts that  
		\begin{itemize}
			\item Rational equivariant algebraic $\K$-groups of number fields vanish in positive even degrees in \Cref{thm:twisted_Borel-1d}
			\item The map in the cofiber \eqref{eqn:MQrho} admits a section, since it is induced by an inclusion of rational $G$-representations, which always has an equivariant cross-section (projection). \qedhere
		\end{itemize}
	\end{proof}
	\bibliographystyle{plain}
	\bibliography{bibliography}
\end{document}